\documentclass[12pt]{amsart}

\usepackage{amssymb}

\usepackage{enumerate}

\usepackage{graphicx}

\makeatletter
\@namedef{subjclassname@2010}{%
  \textup{2010} Mathematics Subject Classification}
\makeatother

\dedicatory{{\rm Appearing in} Acta Arithmetica}

\newtheorem{thm}{Theorem}[section]
\newtheorem{lem}[thm]{Lemma}

\newtheorem{thms}{Theorem}[subsection]
\newtheorem{lems}[thms]{Lemma}
\newtheorem{props}[thms]{Proposition}

\newtheorem*{xcor}{Corollary}

\theoremstyle{definition}

\numberwithin{equation}{section}

\usepackage{color}

\def\(#1;#2){\left\{\begin{array}{l} 
#1 \\ #2
\end{array}\right.}
\def\listtwo(#1;#2;#3;#4){\left\{\begin{array}{cc} 
#1 & #2 \\ 
#3 & #4
\end{array}\right.}			       
\def\listthree(#1;#2;#3;#4;#5;#6){\left\{\begin{array}{ll} 
#1 & #2 \\
#3 & #4 \\
#5 & #6 
\end{array}\right.}		       
\def\smallmattwo(#1;#2;#3;#4){\left(\begin{smallmatrix} 
#1 & #2 \\
#3 &#4
\end{smallmatrix}\right)}
\def\ssmallmattwo(#1;#2;#3;#4){\left.\begin{smallmatrix} 
#1 & #2 \\
#3 &#4
\end{smallmatrix}\right.}			       
\def\mattwo(#1;#2;#3;#4){\left(\begin{array}{cc}
#1 & #2 \\
#3 &#4
\end{array}\right)}
\def\(#1;#2){\left\{\begin{array}{l} 
#1 \\
#2 \end{array}\right.}
\def\vecttwo(#1;#2){\left(\begin{array}{c} 
#1 \\
#2
\end{array}\right)} 
\def\matthree(#1;#2;#3;#4;#5;#6;#7;#8;#9){\left(\begin{array}{ccc} 
#1 & #2 & #3 \\
#4 & #5 &#6 \\
#7 & #8 &#9
\end{array}\right)}			       
\def\smallmatthree(#1;#2;#3;#4;#5;#6;#7;#8;#9){\left(\begin{smallmatrix}{ccc} 
#1 & #2 & #3 \\
#4 & #5 & #6 \\
#7 & #8 & #9
\end{smallmatrix}\right)}		       
\def\vectthree(#1;#2;#3){\left(\begin{array}{c} 
#1 \\
#2 \\
#3
\end{array}\right)}

\newcommand\trans[1]{{}^t \!#1}

\begin{document}


\baselineskip=17pt



\title[Koecher-Maass series of a half-integral weight form]
{Koecher-Maass series of a certain half-integral weight modular form \\
related to the Duke-Imamo{\=g}lu-Ikeda lift}

\author[H. Katsurada]{Hidenori Katsurada}
\address{Muroran Institute of Technology \\ 
Mizumoto 27-1 Muroran, 050-8585, Japan} 
\email{hidenori@mmm.muroran-it.ac.jp}

\author[H. Kawamura]{Hisa-aki Kawamura}
\address{Department of Mathematics, Hokkaido University \\ 
Kita 10, Nishi 8, Kita-Ku, Sapporo, 060-0810, Japan} 
\email{
hkawamura@sci.hokudai.ac.jp}

\date{
}

\begin{abstract}
Let $k$ and $n$ be positive even integers. For a cuspidal Hecke eigenform $h$ in the Kohnen plus space of weight $k-n/2+1/2$ for $\varGamma_0(4),$ let $f$ be the corresponding primitive form of weight $2k-n$ for ${SL}_2({\bf Z})$ under the Shimura correspondence, and $I_n(h)$ the Duke-Imamo{\=g}lu-Ikeda lift of $h$ to the space of cusp forms of weight $k$ for $Sp_n({\bf Z})$. Moreover, let $\phi_{I_n(h),1}$ be the first Fourier-Jacobi coefficient of $I_n(h)$ and $\sigma_{n-1}(\phi_{I_n(h),1})$ be the cusp form  in the generalized Kohnen plus space of weight $k-1/2$ corresponding to $\phi_{I_n(h),1}$ under the Ibukiyama isomorphism. 
We then give an explicit formula for the Koecher-Maass series $L(s,\sigma_{n-1}(\phi_{I_n(h),1}))$ of $\sigma_{n-1}(\phi_{I_n(h),1})$ expressed in terms of the usual $L$-functions of $h$ and $f$. 
\end{abstract}



\maketitle

\section{Introduction} 

Let  $l$ be an integer or a half integer, and let $F$ be a modular form of weight $l$ for the congruence subgroup $\varGamma_0^{(m)}(N)$ of the symplectic group $Sp_m({\bf Z}).$ 
Then the Koecher-Maass series $L(s,F)$ of $F$ is defined as 
 $$L(s,F)=\sum_{A} {c_F(A)  \over e(A)(\det A)^s},$$
where $A$ runs over a complete set of representatives for 
 the $SL_m({\bf Z})$-equivalence classes of positive definite half-integral matrices of degree $m$,
$c_F(A)$ is the $A$-th Fourier coefficient of $F,$ and $e(A)$ denotes the order of the special orthogonal group of $A.$
We note that $L(s, F)$ can also be obtained by the Mellin transform of $F,$ and thus, its analytic properties are relatively known. (As for this, see Maass \cite{M} and Arakawa \cite{Ar1,Ar2,Ar3}.)
Now we are interested in an explicit form of the Koecher-Maass series for a specific choice of $F$. In particular, whenever $F$ is a certain lift of an elliptic modular form $h$ of either integral or half-integral weight, we may hope to express $L(s,F)$ in terms of certain Dirichlet series related to $h$. Indeed, this expectation is verified in the case where $F$ is a lift of $h$ such that the weight $l$ is an integer  (cf. \cite{I-K1,I-K2,I-K3}). In this paper, we discuss a similar problem for 
a lift of elliptic modular forms to half-integral weight Siegel modular forms. 

Let us explain our main result briefly. 
Let $k$ and $n$ be positive even integers.  For a  cuspidal Hecke eigenform $h$ in the Kohnen plus space of weight $k-n/2+1/2$ for $\varGamma_0(4),$  let  $f$ be  the primitive form of weight $2k-n$ for $SL_2({\bf Z})$ corresponding to $h$ under the Shimura correspondence, and let $I_n(h)$ be the Duke-Imamo{\=g}lu-Ikeda lift of $h$ (or of $f$) to the space of  cusp forms of weight $k$ for $Sp_n({\bf Z}).$
We note that $I_2(h)$ is nothing but the Saito-Kurokawa lift of $h.$ 
Let $\phi_{I_n(h),1}$ be the first coefficient of the Fourier-Jacobi expansion of $I_n(h)$ and 
$\sigma_{n-1}(\phi_{I_n(h),1})$ the cusp form in the generalized Kohnen plus space of weight $k-1/2$ 
for   $\varGamma_0^{(n-1)}(4)$ corresponding to $\phi_{I_n(h),1}$ 
 under the Ibukiyama isomorphism $\sigma_{n-1}.$ (For the precise definitions of 
 the Duke-Imamo{\=g}lu-Ikeda lift, the generalized Kohnen plus space and the Ibukiyama isomorphism, see Section 2.)  Then our main result expresses $L(s,\sigma_{n-1}(\phi_{I_n(h),1}))$
 in terms of $L(s,h)$ and $L(s,f)$ (cf. Theorem 2.1).  To prove Theorem 2.1, for a fundamental discriminant $d_0$ and a prime number $p$ we define 
certain formal power series $P_{n-1,p}^{(1)}(d_0,\varepsilon^l,X,t)  \in {\bf C}[X,X^{-1}][[t]]$ associated with some local Siegel series appearing in the $p$-factor of the Fourier coefficient of $\sigma_{n-1}(\phi_{I_n(h),1}).$ Here $\varepsilon$ is the Hasse invariant defined on the set of nondegenerate symmetric matrices with entries in ${\bf Q}_p.$ 
We then rewrite $L(s,\sigma_{n-1}(\phi_{I_n(h),1}))$ in terms of the Euler products $\prod_pP_{n-1,p}^{(1)}(d_0,\varepsilon^l,\beta_p,p^{-s+k/2+n/4-1/4})$ with $l=0,1,$ where $\beta_p$ is the Satake $p$-parameter of $f$  (cf. Theorem 3.2). By using a method similar to those in \cite{I-K2,I-K3}, in Section 4, we get an explicit formula for the formal power series $P_{n-1,p}^{(1)}(d_0,\varepsilon^l,X,t)$ (cf. Theorem 4.4.1), which yields the desired formula for $L(s,\sigma_{n-1}(\phi_{I_n(h),1}))$ immediately. 
The result is very simple, and the 
proof proceeds similarly to the one of \cite{I-K3}, where we gave an explicit formula for the Koecher-Maass series of the Siegel Eisenstein series of integral weight. 
However, it is more elaborate than the preceding one. For instance, we should be careful in dealing with the argument for $p=2$, which divides the level of $\sigma_{n-1}(\phi_{I_n(h),1})$. We also note that the method in this paper is useful for giving an explicit formula for the Rankin-Selberg series of $\sigma_{n-1}(\phi_{I_n(h),1}),$ and as a consequence, we can  prove a conjecture of Ikeda \cite{Ik2} concerning the period of the Duke-Imamo{\=g}lu-Ikeda lift, which we will discuss in \cite{K-K}.  \\

\noindent
{\bf Notation.}  
Let $R$ be a commutative ring. We denote by $R^{\times}$ and $R^*$  the semigroup of non-zero elements of $R$ and the unit group of $R,$  respectively. We also put $S^{\Box}=\{a^2 \, | \, a \in S \}$ for a subset $S$ of  $R.$ We denote by $M_{ml}(R)$ the set of $m \times l$ matrices with entries in $R.$ In particular put $M_m(R)=M_{mm}(R).$ Put $GL_m(R) = \{A \in M_m(R) \, | \, \det A \in R^* \},$ and $SL_m(R)=\{ A \in GL_m(R) \, | \, \det A=1 \},$ where $\det
A$ denotes the determinant of a square matrix $A$. For an $m \times l$ matrix $X$ and an $m \times m$ matrix $A$, we write $A[X] = {}^t X A X,$ where $^t X$ denotes the transpose of $X$. Let $S_m(R)$ denote the set of symmetric matrices of degree $m$ with entries in $R.$ Furthermore, if $R$ is an integral domain of characteristic different from $2,$ let  ${\mathcal L}_m(R)$ denote the set of half-integral matrices of degree $m$ over $R$, that is, ${\mathcal L}_m(R)$ is the subset of symmetric matrices of degree $m$ with entries in the field of fractions of $R$,
whose $(i,j)$-th entry belongs to $R$ or ${1 \over 2}R$ according as $i=j$ or not.  
In particular, we put ${\mathcal L}_m={\mathcal L}_m({\bf Z})$ and ${\mathcal L}_{m,p}={\mathcal L}_m({\bf Z}_p)$ for a prime number $p.$ 
  For a subset $S$ of $M_m(R)$ we denote by $S^{\times}$ the subset of $S$
consisting of all nondegenerate matrices. If $S$ is a subset of $S_m({\bf R})$ with ${\bf R}$, the field of real numbers, we denote by $S_{>0}$ (resp. $S_{\ge 0}$) the subset of $S$
consisting of positive definite (resp. semi-positive definite) matrices. 
The group $GL_m(R)$ acts on the set $S_m(R)$ as $GL_m(R) \times S_m(R) \ni (g,A) \mapsto {}^tg Ag \in S_m(R).$ Let $G$ be a subgroup of $GL_m(R).$ 
For a $G$-stable subset ${\mathcal B}$ of $S_m(R)$, 
we denote by ${\mathcal B}/G$ the set of equivalence classes of ${\mathcal B}$ under the action of $G.$ We sometimes identify ${\mathcal B}/G$ with a complete set of representatives of ${\mathcal B}/G.$ We abbreviate ${\mathcal B}/GL_m(R)$ as ${\mathcal B}/\!\!\sim$ if there is no fear of confusion. For a given ring $R'$, two symmetric matrices $A$ and $A'$ with
entries in $R$ are said to be {\it equivalent over $R'$} to each
other and write $A \sim_{R'} A'$ if there is
an element $X$ of $GL_m(R')$ such that $A'=A[X].$ We also write $A \sim A'$ if there is no fear of confusion. 
For square matrices $X$ and $Y$ we write $X \bot Y = 
\left(\begin{smallmatrix}
X & O \\
O & Y
\end{smallmatrix}\right)$. 

For an integer $D
$ with 
$D \equiv 0 \textrm{ \,or } 1 \text{ mod }4$, 
let $\mathfrak{d}_D$ be the discriminant of ${\bf Q}(\sqrt{D}),$ and put $\mathfrak{f}_D= 
\sqrt{{D \over \mathfrak{d}_D}}.$ 
We call $D$ a {\it fundamental discriminant} if it is either $1$ or the discriminant of some quadratic field extension of ${\bf Q}$. 
For a fundamental discriminant $D,$ let \vspace*{-2mm}$\left(\displaystyle{D \over * }\right)$ be the character corresponding to ${\bf Q}(\sqrt{D})/{\bf Q}.$ 
Here we make the convention that  $\left(\displaystyle{D  \over * } \right)=1$ if $D=1.$ 

We put ${\bf e}(x)=\exp(2 \pi \sqrt{-1} x)$ for $x \in {\bf C}.$ For a prime number $p$ we denote by $\nu_p(*)$ the additive valuation of ${\bf Q}_p$ normalized so that $\nu_p(p)=1,$ and by ${\bf e}_p(*)$ the continuous additive character of ${\bf Q}_p$ such that ${\bf e}_p(x)= {\bf e}(x)$ for $x \in {\bf Q}.$ 

\section{Main  result }

 Put $J_m=
 \left(
 \begin{smallmatrix}
 O_m & -1_m \\
 1_m & O_m
 \end{smallmatrix}
 \right)$, where $1_m$ and $O_m$ denotes the unit matrix and the zero matrix of degree $m$, respectively. 
 Furthermore, put $\varGamma^{(m)}=Sp_m({\bf Z})=\{M \in GL_{2m}({\bf Z})   \, | \,  J_m[M]=J_m \}. $
 Let $l$ be an integer or a half integer. For a congruence subgroup $\varGamma$ of $\varGamma^{(m)},$ we denote by $\mathfrak{M}_{l}(\varGamma)$  the space of holomorphic  modular forms of weight $l$ for $\varGamma.$
We denote by $\mathfrak{S}_{l}(\varGamma)$ the subspace of $\mathfrak{M}_{l}(\varGamma)$ consisting of all cusp forms. 
For a positive integer $N$,
put $\varGamma_0^{(m)}(N)=\left\{\left. 
\left(
 \begin{smallmatrix}
 A & B \\
 C & D
 \end{smallmatrix}
 \right)  \in \varGamma^{(m)} \, \right| \, C \equiv O_m \text{ mod }N \,\right\}$.  
Let $F(Z)$ be an element of $\mathfrak{M}_{l}(\varGamma_0^{(m)}(N)).$ Then 
$F(Z)$ has the following Fourier expansion:
$$F(Z)= \sum_{A \in {({\mathcal L}_{m})}_{\ge 0} } c_F(A) {\bf e}({\rm tr}(AZ)),$$
where ${\rm tr}(X)$ denotes the trace of a matrix $X.$ We then define the Koecher-Maass  series $L(s,F)$ of $F$ as 
$$L(s,F)=\sum_{A \in {({\mathcal L}_{m})}_{>0}/SL_{m}({\bf Z})} {c_F(A)  \over e(A) (\det A)^s},$$ 
where $e(A)=\#\{X \in SL_{m}({\bf Z}) \, | \, A[X]=A \}.$ We note that $L(s,F)$ is nothing but Hecke's $L$-function of $F$ if $m=1$ and $l$ is an integer.

Now put
$${\mathcal L}'_{m}=\{A \in {\mathcal L}_{m}\, | \, A \equiv -\ ^t {r}r  \ {\rm mod} \ 4{\mathcal L}_{m} \text{ for some } r \in {\bf Z}^{m} \}.$$
For $A \in {\mathcal L}'_{m},$ the integral vector $r \in {\bf Z}^{m}$ in the above definition is uniquely determined modulo $2{\bf Z}^{m}$ by $A,$ and is denoted by $r_A.$ Moreover it is easily shown that 
the matrix 
\[\left(\begin{array}{cc}
1&r_A/2 \\
{}^tr_A/2&{({}^tr_Ar_A+A)/4}
\end{array}\right),
\]
which will be denoted by $A^{(1)}$ in the sequel, belongs to ${\mathcal L}_{m+1},$ and that its $SL_{m+1}({\bf Z})$-equivalence class is uniquely determined by $A.$ Suppose that $l$ is a positive even integer.  We define the {\it generalized Kohnen plus space} of weight $l-1/2$ for  $\varGamma_0^{(m)}(4)$ as 
\[
\mathfrak{M}_{l-1/2}^{+}(\varGamma_0^{(m)}(4))=\{ F
\in \mathfrak{M}_{l-1/2}(\varGamma_0^{(m)}(4)) \, |  \,
c_F(A)=0 \text{ unless } A \in {\mathcal L}'_{m}
\},
\]
and put $\mathfrak{S}_{l-1/2}^{+}(\varGamma_0^{(m)}(4))=\mathfrak{M}_{l-1/2}^{+}(\varGamma_0^{(m)}(4)) \cap \mathfrak{S}_{l-1/2}(\varGamma_0^{(m)}(4)).$ 
 Then there exists an isomorphism from  the space of Jacobi forms of index $1$ to  the generalized Kohnen plus space due to Ibukiyama. To explain this, 
let $\varGamma_J^{(m)} = \varGamma^{(m)} \ltimes {\bf {\rm H}}_{m}({\bf Z})$ be the Jacobi group, 
where $\varGamma^{(m)}$ is identified with its image inside $\varGamma^{(m+1)}$ via the natural embedding 
\[\left(
\begin{array}{cc}
A & B \\ 
C & D
\end{array}
\right) \mapsto 
\left(
\begin{array}{cc|cc}
1 & & 0 & \\
 & A &   & B \\ \hline 
 0 & &1& \\
 & C & &D
\end{array} 
\right)\]
and 
\begin{eqnarray*}
\lefteqn{ {\rm H}_m({\bf Z})=} \\
&& \left\{ 
\left.
\left(
\begin{array}{cc|cc}
  1 & 0 & \kappa & \mu  \\
  0 & 1_m & \trans{\mu} & O_m  \\ \hline
  {} & {} & 1 & 0  \\
  {} & {} & 0 & 1_m  \\
\end{array}
\right)
\left(
\begin{array}{cc|cc}
  1 & \lambda & {} & {}  \\
  0 & 1_m & {} & {}  \\ \hline
  {} & {} & 1 & 0  \\
  {} & {} & -\trans{\lambda} & 1_m  \\
\end{array}
\right) 
\, \right| \! 
\begin{array}{l}
(\lambda,\mu) \in {\bf Z}^m \oplus {\bf Z}^m, \\[1mm]
  \,\, \kappa \in {\bf Z} 
\end{array}
\!
\right\}.
\end{eqnarray*} 
Let $J_{l,\, 1}(\varGamma_J^{(m)})$ denote the space of Jacobi  forms of weight $l$ and index $1$ for $\varGamma_J^{(m)}$,
and $J_{l,\, 1}^{{\rm cusp}}(\varGamma_J^{(m)})$ the subspace of $J_{l,\, 1}(\varGamma_J^{(m)})$ consisting of all 
cusp forms. 
Let $\phi(Z,z) \in J_{l,\, 1}(\varGamma_J^{(m)}).$ 
Then we  have the following Fourier  expansion:
$${\phi}(Z,\,z)= 
\displaystyle\sum_{\scriptstyle T \in {\mathcal L}_{m},\, r \in {\bf Z}^{m}, \atop {\scriptstyle 4T-{}^t {r}r \ge 0}}
c_{\phi}(T,r){\bf e}({\rm tr}(T Z)+r^t{z}). 
$$
 We then define  $\sigma_m(\phi)$ as
$$\sigma_{m}(\phi)=\sum_{A \in  {({\mathcal L}_{m}')}_{\ge 0}}c_{\phi}((A+{}^tr_Ar_A)/4,r_A){\bf e}({\rm tr}(AZ)),$$
where $r=r_A$ denotes an element of ${\bf Z}^{m}$ such that $A+{}^tr_Ar_A \in 4 {\mathcal L}_{m}.$ This $r_A$ is uniquely determined modulo $2{\bf Z}^{m},$ and $c_{\phi}((A+{}^tr_Ar_A)/4,r_A)$ does not depend on the choice of the representative of $r_A \ {\rm mod} \ 2{\bf Z}^{m}.$ 
Then Ibukiyama \cite{Ib} showed that the mapping $\sigma_{m}$ gives 
a ${\bf C}$-linear isomorphism 
$
J_{l,\, 1}(\varGamma_J^{(m)}) \simeq
 \mathfrak{M}_{l-1/2}^{+}(\varGamma_0^{(m)}(4)),$ and in particular,  
$\sigma_{m}({J_{l,\, 1}^{\, {\rm cusp}}}(\varGamma_J^{(m)}))= \mathfrak{S}_{l-1/2}^{+}(\varGamma_0^{(m)}(4)).$
We call $\sigma_{m}$ the {\it Ibukiyama isomorphism}.
\bigskip

Let $p$ be a prime number.  For a non-zero element $a \in {\bf Q}_p$ we put $\chi_p(a)=1,-1,$ or
 $0$ according as ${\bf Q}_p(a^{1/2})={\bf Q}_p, {\bf Q}_p(a^{1/2})$ is
 an unramified quadratic extension of ${\bf Q}_p,$ or ${\bf Q}_p(a^{1/2})$
 is a ramified quadratic extension of ${\bf Q}_p.$
 We note that  $\chi_p(D)=\left({\displaystyle D \over \displaystyle p }\right)$ if $D$ is a fundamental discriminant. For the rest of this section, let $n$ be a positive even  integer.
  For an element  $T$ of ${\mathcal L}_{n,p}^{\times},$ put
 $\xi_p(T)=\chi_p((-1)^{n/2} \det T).$ 
  Let $T$ be an element of ${\mathcal L}_n^{\times}.$  Then  $(-1)^{n/2} \det (2T) \equiv 0 $ or $1 \ {\rm mod} \ 4,$ and  we define $\mathfrak{d}_T$ and $\mathfrak{f}_T$ as $\mathfrak{d}_T =\mathfrak{d}_{(-1)^{n/2}\det (2T)}$ and $\mathfrak{f}_T =\mathfrak{ f}_{(-1)^{n/2}\det (2T)},$ respectively. 
 For an element $T$ of ${\mathcal L}_{n,p}^{\times},$ there exists an element $\widetilde T$ of ${\mathcal L}_n^{\times}$ such that $\widetilde T \sim_{{\bf Z}_p} T.$ We then put $\mathfrak{e}_p(T)= \nu_p(\mathfrak{f}_{\widetilde T}),$ and $[\mathfrak{d}_T]=\mathfrak{d}_{\widetilde T} \ {\rm mod} \ {{\bf Z}_p^*}^{\Box}.$ They do not depend on the choice of $\widetilde T.$  We note that $(-1)^{n/2} \det (2T)$ can be expressed as $(-1)^{n/2} \det (2T) =dp^{2\mathfrak{e}_p(T)} \ {\rm mod} \ {{\bf Z}_p^*}^{\Box}$ for any $d \in [\mathfrak{d}_T].$ 

For each $T \in {\mathcal L}_{n,p}^{\times}$ we define the {\it local Siegel series} $b_p(T,s)$ by 
  $$b_p(T,s)=\sum_{R \in S_n({\bf Q}_p)/S_n({\bf Z}_p)} {\bf e}_p({\rm tr}(TR))p^{-\nu_p(\mu_p(R))s},$$  
where $\mu_p(R)=[R{\bf Z}_p^n+{\bf Z}_p^n:{\bf Z}_p^n].$ 
 We remark that there exists a unique polynomial 
 $F_p(T,X)$ in $X$ such that 
 $$b_p(T,s)=F_p(T,p^{-s}){(1-p^{-s})\prod_{i=1}^{n/2} (1-p^{2i-2s}) \over 1-\xi_p(T)p^{n/2-s}}$$ 
(cf. Kitaoka \cite{Ki1}). 
We then define a polynomial $\widetilde F_p(T,X)$ in $X$ and $X^{-1}$ as
$$\widetilde F_p(B,X)=X^{-\mathfrak{e}_p(T)}F_p(T,p^{-(n+1)/2}X).$$
We remark that  $\widetilde{F}_p(B,X^{-1})=\widetilde{F}_p(B,X)$ (cf. \cite{Kat1}). \vspace*{2mm}
Now, for a positive even integer $k$, let 
 $$h(z)=\sum_{\scriptstyle m \in {\bf Z}_{>0}, \atop {\scriptstyle (-1)^{n/2}m \equiv 0, 1 \, {\rm mod}\, 4 }}c_h(m){\bf e}(mz)$$
  be a Hecke eigenform in the Kohnen plus space $\mathfrak{S}_{k-n/2+1/2}^+(\varGamma_0(4))$ and 
$$f(z)=\sum_{m=1}^{\infty}c_f(m){\bf e}(mz)$$
 the primitive form in $\mathfrak{S}_{2k-n}(\varGamma^{(1)})$ corresponding to $h$ under the  Shimura correspondence (cf. Kohnen \cite{Ko}). Let $\beta_p \in {\bf C}$ such that $\beta_p+\beta_p^{-1}=p^{-k+n/2+1/2}c_f(p),$ which we call the {\it Satake $p$-parameter} of $f$. 
We define a Fourier series $I_n(h)(Z)$ on 
${\bf H}_n$ by
$$I_n(h)(Z)= \sum_{T \in {({\mathcal L}_n)}_{> 0}} c_{I_n(h)}(T){\bf e}({\rm tr}(TZ)),$$
 where
 $$c_{I_n(h)}(T)=c_h(|\mathfrak{d}_T|) \mathfrak{f}_T^{k-n/2-1/2} \prod_p\widetilde F_p(T,\beta_p).$$ 
Then Ikeda \cite{Ik1} showed  that 
$I_n(h)(Z)$ is a Hecke eigenform in $\mathfrak{S}_k(\varGamma^{(n)})$ whose 
standard $L$-function coincides with $\zeta(s)\prod_{i=1}^n L(s+k-i,f),$ where $\zeta(s)$ is Riemann's zeta function. 
 The existence of such a Hecke eigenform was conjectured by Duke and Imamo{\=g}lu in their unpublished paper. 
  We call $I_n(h)$ the {\it Duke-Imamo{\=g}lu-Ikeda lift} of $h$ (or of $f$) as in Section 1. Let $\phi_{I_n(h),1}$ be the first coefficient of the Fourier-Jacobi expansion of $I_n(h),$ that is, 
\[
I_n(h)\!\left(\left(
\begin{array}{cc}
w & z \\ 
{}^t z & Z
\end{array}
\right)\right) = \sum_{N=1}^{\infty} \phi_{I_n(h),N}(Z,\,z){\bf e}(Nw), \]  
where $Z \in {\bf H}_{n-1}$, $z \in {\bf C}^{n-1}$ and $w \in {\bf H}_1$. We easily see that $\phi_{I_n(h),1}$ belongs to $J_{k-1/2,\, 1}^{{\rm cusp}}(\varGamma_J^{(n-1)})$ and 
  $$\phi_{I_n(h),1}(Z,\,z)= 
\displaystyle\sum_{\scriptstyle T \in {\mathcal L}_{n-1},\, r \in {\bf Z}^{n-1}, \atop {\scriptstyle 4T-{}^t {r}r > 0}}
c_{I_n(h)}({\smallmattwo(1;r/2;{}^tr/2;T)}){\bf e}({\rm tr}(T Z)+r^t{z}). 
$$
Moreover we have 
 $$\sigma_{n-1}( \phi_{I_n(h),1})(Z)=\sum_{T \in {({\mathcal L}_{n-1}')}_{>0}} c_{I_n(h)}(T^{(1)}){\bf e}({\rm tr}(T Z)).$$ 
Put $\Gamma_{{\bf C}}(s)=2(2\pi)^{-s} \Gamma(s),$
and $\widetilde{\xi}(s) = \Gamma_{{\bf C}}(s) \zeta(s)$. 
Then our main result in this paper is stated as follows:

\begin{thm}
Let $h$ and $f$ be as above.  Then we have
\begin{eqnarray*}
\lefteqn{L(s,\sigma_{n-1}(\phi_{I_n(h),1}))=2^{-\delta_{2,n}-s(n-2)-(n-2)/2}\prod_{i=1}^{(n-2)/2} \widetilde \xi(2i) } \\
&& \times \left\{L(s-n/2+1,h)\prod_{i=1}^{(n-2)/2} L(2s-n+2i+1,f) \right. \\
&& \hspace*{10mm}+\left.(-1)^{n(n-2)/8}L(s,h)\prod_{i=1}^{(n-2)/2} L(2s-n+2i,f)\right\},
\end{eqnarray*}
where $\delta_{2,n}$ denotes Kronecker's delta. 
\end{thm}
In the case of $n=2,$ the modular form $\sigma_{n-1}(\phi_{I_n(h),1})$ is $h$ itself,   and then the above formula is trivial.
We note that, unlike the cases of \cite{I-K1,I-K2,I-K3}, there does not appear any convolution product of modular forms in the above theorem.
However, the proof may not be so simple because 
the nature of Fourier coefficients of the modular form 
$\sigma_{n-1}(\phi_{I_n(h),1})$ is much more complicated rather than those in the papers cited above. 

\section{Reduction to local computations}
By definition, it turns out that the Fourier coefficient of $\sigma_{n-1}(\phi_{I_n(h),1})$ can be expressed as a product of local Siegel series, and therefore  we can reduce the problem to local computations. To explain this, we recall some terminologies and notation.
For given $a,b \in {\bf Q}_p^{\times}$ let $(a,b)_p$ denote the Hilbert symbol over ${\bf Q}_p.$ 
Following Kitaoka \cite{Ki2}, we define the {\it Hasse invariant} $\varepsilon(A)$ of $A \in S_m({\bf Q}_p)^{\times}$ by 
$$\varepsilon(A)=\prod_{1 \le i \le j \le m}(a_i,a_j)_p$$
if $A$ is equivalent to $a_1 \bot \cdots \bot a_m$ over ${\bf Q}_p$ with some $a_1,a_2,...,a_m \in {\bf Q}_p^{\times}.$ We note that this definition does not depend on the choice of $a_1,a_2,...,a_m.$ 

Now put
$${\mathcal L}_{m,p}'=\{ A \in {\mathcal L}_{m,p} \, | \, A \equiv -\ ^t {r}r  \ {\rm mod} \ 4{\mathcal L}_{m,p} \ { \rm for \ some } \ r \in {\bf Z}_p^{m} \} .$$
Furthermore we put  $S_{m}({\bf Z}_p)_e=2{\mathcal L}_{m,p}$ and $S_{m}({\bf Z}_p)_o=S_{m}({\bf Z}_p) \setminus S_{m}({\bf Z}_p)_e$. We note that 
 ${\mathcal L}_{m,p}'={\mathcal L}_{m,p}=S_{m}({\bf Z}_p)$ if $p\not=2.$ Let $T \in {\mathcal L}_{m-1,p}'.$ Then there exists an element $r \in {\bf Z}_p^{m-1}$  such that $\smallmattwo(1;r/2;{}^tr/2;{(T+{}^trr)/4})\in 
 {\mathcal L}_{m,p}.$  As is easily shown, $r$ is uniquely determined by $T$, modulo $2{\bf Z}_p^{m-1},$ and is denoted by $r_T.$ Moreover as will be shown in the next lemma, $\smallmattwo(1;r_T/2;{}^tr_T/2;{(T+{}^tr_Tr_T)/4})$ is uniquely determined by $T$, up to $GL_m({\bf Z}_p)$-equivalence, and is denoted by $T^{(1)}.$ 


\begin{lem}
Let $m$ be a positive integer. 
\begin{enumerate}
\item[{\rm (1)}] 
Let $A$ and $B$ be elements of ${\mathcal L}_{m-1,p}'.$ 
Then $\smallmattwo(1;r_A/2;{}^tr_A/2;{(A+{}^tr_Ar_A)/4}) \sim \smallmattwo(1;r_B/2;{}^tr_B/2;{(B+{}^tr_Br_B)/4})$ if $A \sim B.$
\item[{\rm (2)}] Let $A \in {\mathcal L}_{m-1,p}'.$ 
\begin{enumerate}
\item[{\rm (2.1)}]  Let $p\not=2.$ Then  
$A^{(1)}  \sim \mattwo(1;0;0;A).$

\item[{\rm (2.2)}] {
Let $p=2.$ If  $r_A \equiv 0 \ {\rm mod} \ 2,$ then 
$A \sim 4B$ with $B \in {\mathcal L}_{m-1,2},$ and 
$ A^{(1)}  \sim \mattwo(1;0;0;B).$ In particular, $\nu_2((\det B)) \ge m$ or $m+1$ according as $m$ is even or odd. 
If  $r_A \not\equiv 0 \ {\rm mod} \ 2,$ then 
$A \sim a \bot 4B$ with $a \equiv -1 \ {\rm mod} \ 4$ and $B \in {\mathcal L}_{m-2,2}$ and  we have
$A^{(1)}  \sim \matthree(1;1/2;0;1/2;{(a+1)/4};0;0;0; B).$ In particular, $\nu_2((\det B)) \ge m$ or $m-1$ according as $m$ is even or odd.
}
\end{enumerate}
\end{enumerate}
\end{lem}
\begin{proof} The assertion can easily be proved.
\end{proof}

%
Now let $m$ be a positive even integer. For $T \in ({\mathcal L}'_{m-1})^{\times},$ put  $\mathfrak{d}_T^{(1)}=\mathfrak{d}_{T^{(1)}},$ and $\mathfrak{f}_T^{(1)}=\mathfrak{f}_{T^{(1)}},$ and for $T \in ({\mathcal L}_{m-1,p}')^{\times},$ we define $[\mathfrak{d}_T^{(1)}]$ and $\mathfrak{e}^{(1)}_T$ as $[\mathfrak{d}_{T^{(1)}}]$ and $\mathfrak{e}_{T^{(1)}},$ respectively. These do not depend on the choice of $r_T.$ 
We note that  $(-1)^{m/2} \det T= 2^{m-2}{\mathfrak{f}_T^{(1)}}^2 \mathfrak{d}_T^{(1)}$ for $T \in  ({\mathcal L}'_{m-1})^{\times}.$
We define a polynomial $F_p^{(1)}(T,X)$ in $X,$ and a polynomial $\widetilde F_p^{(1)}(T,X)$ in $X$ and $X^{-1}$ by 
$$F_p^{(1)}(T,X)=F_p(T^{(1)}, X), $$ 
and 
$$\widetilde F_p^{(1)}(T,X)= X^{-\mathfrak{e}_p^{(1)}(T)} F_p^{(1)}(T,p^{-(m+1)/2}X).$$
Let $B$ be  an element of   $({\mathcal L}_{m-1,p}')^{\times}.$ Let $p \not=2.$ Then 
 $$\widetilde F_p^{(1)}(B,X)=\widetilde F_p(1 \bot B,X).$$
 Let $p=2.$ Then  
\begin{eqnarray*}
\lefteqn{
\widetilde{F}_2^{(1)}(B,X)
} \\
&=&\left\{
\begin{array}{cl}
\widetilde F_2({\mattwo(1;1/2;1/2;{(a+1)/4})} \bot B',X) &  
{
\begin{array}{ll}
{\rm if} \ B=a \bot 4B'  \  {\rm with}\\
\hspace{0.5cm} \ a \equiv -1 \ {\rm mod} \ 4, B' \in {\mathcal L}_{m-2,2},
\end{array}
} \\[5mm]
\widetilde F_2(1 \bot B',X) 
& {\rm if} \ B=4B' \  {\rm with} \ B' \in {\mathcal L}_{m-1,2}.
\end{array}
\right.
\end{eqnarray*}

Now let $m$ and $l$ be positive integers such that $m \ge l.$ Then for nondegenerate  symmetric matrices $A$ and  $B$ of degree $m$ and $l$ respectively with entries in ${\bf Z}_p$  we define the {\it local density} $\alpha_p(A,B)$ and the {\it primitive local density} $\beta_p(A,B)$ representing $B$ by $A$ as
$$\alpha_p(A,B)=2^{-\delta_{m,l}}\lim_{a \rightarrow
\infty}p^{a(-ml+l(l+1)/2)}\#{\mathcal A}_a(A,B)$$
 and
 $$\beta_p(A,B)=2^{-\delta_{m,l}}\lim_{a \rightarrow
\infty}p^{a(-ml+l(l+1)/2)}\#{\mathcal B}_a(A,B),$$
where $${\mathcal A}_a(A,B)=\{X \in
M_{ml}({\bf Z}_p)/p^aM_{ml}({\bf Z}_p) \, | \, A[X]-B \in p^aS_l({\bf Z}_p)_e \}$$
and 
$${\mathcal B}_a(A,B)=\{X \in {\mathcal A}_a(A,B) \, | \, 
  {\rm rank}_{{\bf Z}_p/p{\bf Z}_p} (X \ {\rm mod} \ p) =l \}.$$
In particular we write $\alpha_p(A)=\alpha_p(A,A).$ 
 Put  
$${\mathcal F}_{p}=\{d_0 \in {\bf Z}_p \, | \, \nu_p(d_0) \le 1\}$$ 
 if $p$ is an odd prime, and  
$${\mathcal F}_{2}=\{d_0 \in {\bf Z}_2 \, | \,  d_0 \equiv 1 \ {\rm mod} \ 4, \ {\rm  or} \  d_0/4  \equiv -1 \  {\rm mod} \ 4,  \ {\rm or} \ \nu_2(d_0)=3 \}.$$
Let $m$ be a positive integer. For $d \in {\bf Z}_p^{\times}$ put 
 \begin{eqnarray*}
\lefteqn{S_{m}({\bf Z}_p,d)}\\
&=&\{T \in S_{m}({\bf Z}_p) \ |\ (-1)^{[(m+1)/2]} \det T=p^{2i}d \ {\rm mod} \ {{\bf Z}_p^*}^{\Box} \ {\rm with \ some \, } \ i \in {\bf Z} \}, 
\end{eqnarray*}
and $S_{m}({\bf Z}_p,d)_x=S_{m}({\bf Z}_p,d) \cap S_{m}({\bf Z}_p)_x$ for $x=e$ or $o.$   
Put ${\mathcal L}_{m,p}^{(0)}=S_{m}({\bf Z}_p)_e^{\times}$ and ${\mathcal L}_{m,p}^{(1)}=({\mathcal L}_{m,p}')^{\times}.$   We also define ${\mathcal L}_{m,p}^{(j)}(d)= S_{m}({\bf Z}_p,d) \cap {\mathcal L}_{m,p}^{(j)}$  for $j=0,1.$ 
Let $\iota_{m,p}$ be the constant function on ${\mathcal L}_{m,p}^{\times}$ taking the value 1, and $\varepsilon_{m,p}$ the function on ${\mathcal L}_{m,p}^{\times}$ assigning the Hasse invariant of $A$ for $A \in {\mathcal L}_{m,p}^{\times}.$  We sometimes drop the suffix and write $\iota_{m,p}$ as $\iota_p$ or $\iota$ and the others if there is no fear of confusion.
From now on, we sometimes write $\omega=\varepsilon^l$ with $l=0$ or $1$ according as $\omega=\iota$ or $\varepsilon.$ 
 Let $n$ be a positive even  integer.  For $d_0 \in {\mathcal F}_{p}$ and  $\omega=\varepsilon^l$ with $l=0$ or $1$, we  define a  formal power series $P_{n-1,p}^{(1)}(d_0,\omega,X,t)$ in $t$ by 
$$P_{n-1,p}^{(1)}(d_0,\omega,X,t)=\kappa(d_0,n-1,l)^{-1} t^{\delta_{2,p}(2-n)}\hspace*{-2.5mm}\sum_{B \in {\mathcal L}_{n-1,p}^{(1)}(d_0)}  \hspace*{-2.5mm}{\widetilde F_p^{(1)}(B,X) \over \alpha_p(B)}\omega(B)t^{\nu(\det B)},$$
where 
\begin{eqnarray*}
\lefteqn{\kappa(d_0,r-1,l)= \kappa(d_0,r-1,l)_p} \\
&&=\{(-1)^{lr(r-2)/8} 2^{-(r-2)(r-1)/2}\}^{\delta_{2,p}} \cdot 
((-1)^{r/2},(-1)^{r/2}d_0 )_p^l\,\, p^{-(r/2-1)l\nu(d_0)} 
\end{eqnarray*}
 for a positive even integer $r.$ 
This type of formal power series appears in an explicit formula for the Koecher-Maass series associated with 
the Siegel Eisenstein series and the Duke-Imamo{\=g}lu-Ikeda lift (cf. \cite{I-K2}, \cite{I-K3}). Therefore we say that this formal power series is {\it of Koecher-Maass type}. 
An explicit formula for 
$P_{n-1,p}^{(1)}(d_0,\omega,X,t)$  will be given in the next section.  
Let ${\mathcal F}$ denote the set of fundamental discriminants, and for $l=\pm 1,$ put 
$${\mathcal F}^{(l)}=\{ d_0 \in {\mathcal F} \  | \ ld_0 >0 \}.$$
Now let $h$ be a Hecke eigenform in $\mathfrak{S}_{k-n/2+1/2}^+(\varGamma_0(4)),$ and $ f,I_n(h),\phi_{I_n(h),1}$ and $\sigma_{n-1}(\phi_{I_n(h),1})$ be as in Section 2.

\begin{thm}
 Let the notation and the assumption be as above. Then for ${\rm Re}(s) \gg 0,$ we have 
\begin{eqnarray*}
\lefteqn{L(s,\sigma_{n-1}(\phi_{I_n(h),1}))=\kappa_{n-1}2^{-(n-2)s -(n-2)/2 -\delta_{2,n}} } \\[2mm]
&& \hspace*{-5mm}\times 
\left\{\sum_{d_0 \in {\mathcal F}^{((-1)^{n/2})} }c_h(|d_0|)|d_0|^{n/4-k/2+1/4} \prod_p   P_{n-1,p}^{(1)}(d_0,\iota_p,\beta_p,p^{-s+k/2+n/4-1/4}) \right. \\[2mm]
&& \hspace*{5mm}
+ (-1)^{n(n-2)/8} \hspace*{-2mm}\sum_{d_0 \in {\mathcal F}^{((-1)^{n/2})} }c_h(|d_0|)|d_0|^{-n/4-k/2+5/4} \\
&& \hspace*{45mm} \times \left.\prod_p P_{n-1,p}^{(1)}(d_0,\varepsilon_p,\beta_p,p^{-s+k/2+n/4-1/4})\right\},
\end{eqnarray*}
where $\kappa_{n-1}=\prod_{i=1}^{(n-2)/2}\Gamma_{\bf C}(2i).$
\end{thm}
\begin{proof}
Let $T \in {({\mathcal L}'_{n-1})}_{>0}.$ Then it follows from Lemma 4.1 that the $T$-th Fourier coefficient $c_{\sigma_{n-1}(\phi_{I_n(h),1})}(T)$  of $\sigma_{n-1}(\phi_{I_n(h),1})$ is uniquely determined by the genus to which $T$ belongs, and by definition, it can be expressed as
$$c_{\sigma_{n-1}(\phi_{I_n(h),1})}(T)=c_{I_n(h)}(T^{(1)})=c_h(|\mathfrak{d}_T^{(1)}|)(\mathfrak{f}_T^{(1)})^{k-n/2-1/2}\prod_p \widetilde F^{(1)}(T,\beta_p).$$
We note that 
$$  (\mathfrak{f}_T^{(1)})^{k-n/2-1/2}=|\mathfrak{d}_T^{(1)}|^{-(k/2-n/4-1/4)} (\det T)^{(k/2-n/4-1/4)}2^{-(n-2)(k/2-n/4-1/4)} $$
for
$T \in {({\mathcal L}_{n-1})}_{>0}.$ 
We also note that  
$$\sum_{T' \in {\mathcal G}(T)} {1 \over e(T')}=\kappa_{n-1} 2^{3-n-\delta_{2,n}}\det T^{n/2} \prod_p \alpha_p(A)^{-1}$$
for $T \in {S_{n-1}({\bf Z})}_{>0}$, where ${\mathcal G}(T)$ denotes the set of $SL_{n-1}({\bf Z})$-equivalence classes belonging to the genus of $T$
(cf. Theorem 6.8.1 in \cite{Ki2}
). Hence we have
$$\sum_{T' \in {\mathcal G}(T)} {c_{\sigma_{n-1}(\phi_{I_n(h),1})}(T) \over e(T')}=\kappa_{n-1} 2^{3-n-\delta_{2,n} -(n-2)(k/2-n/4-1/4)}$$
$$ \times \det T^{k/2+n/4-1/4} |\mathfrak{d}_T^{(1)}|^{-k/2+n/4+1/4}\prod_p {\widetilde F_p^{(1)}(T,\beta_p) \over  \alpha_p(T)}.$$
Thus, by using the same method as in Proposition 2.2 of \cite{I-S}, similarly to Theorem 3.3 (1) of \cite{I-K1}, and Theorem 3.2 of \cite{I-K2},  we obtain
 \begin{eqnarray*}
\lefteqn
{L(s,\sigma_{n-1}(\phi_{I_n(h),1}))
} \\[2mm]
&=&\kappa_{n-1}2^{-(k/2-n/4-1/4)(n-2)+2-n-\delta_{2,n}}
\sum_{d_0 \in {\mathcal F}^{((-1)^{n/2})} }c_h(|d_0|)|d_0|^{n/4-k/2+1/4} \\
&\times&\hspace*{-2.5mm} 
\left\{ 2^{(-s+k/2+n/4-1/4)(n-2)}) \prod_p  \kappa_p(d_0,n-1,0)  P_{n-1,p}^{(1)}(d_0,\iota_p,\beta_p,p^{-s+k/2+n/4-1/4}) \right. \\
&+ & \left. 2^{(-s+k/2+n/4-1/4)(n-2)}  \prod_p \kappa_p(d_0,n-1,1) P_{n-1,p}^{(1)}(d_0,\varepsilon_p,\beta_p,p^{-s+k/2+n/4-1/4})\right\}. 
\end{eqnarray*}
We note that $\displaystyle \prod_p((-1)^{n/2},(-1)^{n/2}d_0)_p=1.$ Hence
$\prod_p  \kappa_p(d_0,n-1,0) =2^{-(n-2)(n-1)/2},$ and $\prod_p  \kappa_p(d_0,n-1,0) =(-1)^{n(n-2)/8}|d_0|^{-n/2+1} 2^{-(n-2)(n-1)/2}.$
This proves the assertion. 
\end{proof}

\section{Formal power series associated with local Siegel series}

Throughout this section we fix a positive even integer $n.$ 
We also simply write $\nu_p$ as $\nu$ and the others if the prime number $p$ is clear from the context.
In this section,  we give an explicit formula for $P_{n-1}^{(1)}(d_0,\omega,X,t)$ with $\omega=\varepsilon^l \ (l=0,\,1)$ to prove Theorem 3.2 (cf. Theorem 4.4.1). 
The idea of the proof is to express the power series in question as a sum of certain subseries (cf.  Proposition 4.4.3). 
Henceforth, for a $GL_m({\bf Z}_p)$-stable subset ${\mathcal B}$ of $S_m({\bf Q}_p),$ we simply write  $\sum_{T \in {\mathcal B}}$ instead of $\sum_{T \in {\mathcal B}/\sim}$ if there is no fear of confusion. 

\subsection{Local densities}

\noindent
{ }

\bigskip

Put ${\mathcal D}_{m,i}=GL_m({\bf Z}_p) \mattwo(1_{m-i};0;0;p1_i) GL_m({\bf Z}_p)$. For two elements $S$ and $T$ of $S_m({\bf Z}_p)^{\times}$  and an nonnegative 
integer $i \le m,$ we define $\alpha_p(S,T,i)$ as 
$$\alpha_p(S,T,i)=2^{-1}\lim_{e \to \infty}p^{-e(m-1)m/2} {\mathcal A}_e(S,T,i),$$
where
$${\mathcal A}_e(S,T,i)=\{ \overline {X} \in {\mathcal A}_e(S,T) \, | \, X \in {\mathcal D}_{m,i} \}.$$

\begin{lems}
Let $S$ and $T$ be elements of $S_m({\bf Z}_p)^{\times}.$ \\
{\rm (1)} Let $\Omega(S,T)=\{W \in M_m({\bf Z}_p)^{\times} \, | \, S[W] \sim T \},$ and $\Omega(S,T,i)= \Omega(S,T) \cap  {\mathcal D}_{m,i}. $ Then 
$${\alpha_p(S,T) \over \alpha_p(T)}=\#(\Omega(S,T)/GL_m({\bf Z}_p))p^{-m(\nu(\det T)- \nu(\det S))/2},$$
and
$${\alpha_p(S,T,i) \over \alpha_p(T)}=\#(\Omega(S,T,i)/GL_m({\bf Z}_p))p^{-m(\nu(\det T)- \nu(\det S))/2} .$$
\noindent
{\rm (2)} Let $\widetilde \Omega(S,T)=\{W \in M_m({\bf Z}_p)^{\times} \, | \, S \sim T[W^{-1}] \},$ and $\widetilde \Omega(S,T,i)= \widetilde \Omega(S,T) \cap {\mathcal D}_{m,i}.$ Then   
$${\alpha_p(S,T) \over \alpha_p(S)}=\#(GL_m({\bf Z}_p) \backslash \widetilde \Omega(S,T))p^{(\nu(\det T)- \nu(\det S))/2},$$
and 
$${\alpha_p(S,T,i) \over \alpha_p(S)}=\#(GL_m({\bf Z}_p) \backslash \widetilde \Omega(S,T,i))p^{(\nu(\det T)- \nu(\det S))/2}.$$
\end{lems}


\begin{proof} The assertion (1) follows from Lemma 2.2 of \cite{B-S}. Now  by Proposition 2.2 of \cite{Kat1} we have
$$\alpha_p(S,T)=\sum_{W \in GL_m({\bf Z}_p) \backslash \widetilde \Omega(S,T)} \beta_p(S,T[W^{-1}])p^{\nu(\det W)}.$$
Then  
$\beta_p(S,T[W^{-1}])=\alpha_p(S)$ or $0$ according as $S \sim T[W^{-1}]$ or not.
Thus the assertion (2) holds.
\end{proof}

A nondegenerate ${m \times m}$ matrix $D=(d_{ij})$ with entries in ${\bf Z}_p$ is said to be {\it reduced} if $D$ satisfies the following two conditions:\begin{enumerate}
\item[(a)] For 
$i=j$, $d_{ii}=p^{e_{i}}$ with a nonnegative integer $e_i$; \vspace*{1mm}
 
\item[(b)] For 
$i\ne j$, $d_{ij}$ is a nonnegative integer satisfying 
$ d_{ij} \le p^{e_j}-1$ if $i <j$ and $d_{ij}=0$ if $i >j$. 
\end{enumerate}
It is well known that  we can take the set of all reduced matrices as a  complete set of representatives of $GL_m({\bf Z}_p) \backslash M_m({\bf Z}_p)^{\times}.$  Let $j=0$ or $1$ according as $m$ is even or odd.   For $B \in {\mathcal L}_{m,p}^{(j)}$ put
$$\widetilde \Omega^{(j)}(B)=\{W \in M_m({\bf Z}_p)^{\times} \, | \, B[W^{-1}] \in {\mathcal L}_{m,p}^{(j)} \}.$$
Furthermore put $\widetilde \Omega^{(j)}(B,i)=\widetilde \Omega^{(j)}(B) \cap {\mathcal D}_{m.i}.$ Let $n_0 \le m,$ and $\psi_{n_0,m}$ be the mapping from $GL_{n_0}({\bf Q}_p)$ into $GL_{m}({\bf Q}_p)$ defined by $\psi_{n_0,m}(D)=1_{m-n_0} \bot D.$ 

\bigskip

\begin{lems}
{\rm (1)} Suppose that $p \not= 2.$ Let $\Theta \in GL_{n_0}({\bf Z}_p) \cap S_{n_0}({\bf Z}_p),$ and  $B_1 \in S_{m-n_0}({\bf Z}_p)^{\times}.$ 
\begin{enumerate}
\item[{\rm (1.1)}] Let $n_0$ be even.
Then  $\psi_{m-n_0,m}$ induces a bijection  
\[GL_{m-n_0}({\bf Z}_p) \backslash \widetilde \Omega^{(j)}(B_1)\simeq 
GL_{m}({\bf Z}_p) \backslash \widetilde \Omega^{(j)}(\Theta \bot B_1),\]
where $j=0$ or $1$ according as $m$ is even or odd. In particular, we have 
\[GL_{m-n_0}({\bf Z}_p) \backslash \widetilde \Omega^{(j)}(pB_1)\simeq 
GL_{m}({\bf Z}_p) \backslash \widetilde \Omega^{(j)}(\Theta \bot pB_1).\]

\item[{\rm (1.2)}] Let $n_0$ be odd.
Then $\psi_{m-n_0,m}$ induces a bijection 
\[GL_{m-n_0}({\bf Z}_p)\backslash \widetilde \Omega^{(j')} (B_1) 
\simeq GL_{m}({\bf Z}_p) \backslash \widetilde \Omega^{(j)}(\Theta \bot B_1),\]
where $j=0$ or $1$ according as $m$ is  even or odd, and $j'=1$ or $0$ according as $m$ is even or odd.
In particular, we have 
\[GL_{m-n_0}({\bf Z}_p)\backslash \widetilde \Omega^{(j')} (pB_1) 
\simeq GL_{m}({\bf Z}_p) \backslash \widetilde \Omega^{(j)}(\Theta \bot pB_1),\]
\end{enumerate}
\noindent
{\rm (2)} Suppose that $p=2.$ Let $m$ be a positive integer, $n_0$  
an even integer not greater than $m,$ and  $\Theta \in GL_{n_0}({\bf Z}_2) \cap S_{n_0}({\bf Z}_2)_e.$ 
\begin{enumerate}
\item[{\rm (2.1)}] Let   $B_1 \in S_{m-n_0}({\bf Z}_2)^{\times}.$ Then $\psi_{m-n_0,m}$ induces a bijection 
\[\hspace*{5mm}GL_{m-n_0}({\bf Z}_2) \backslash \widetilde \Omega^{(j)}(2^{j+1}B_1)\simeq
GL_{m}({\bf Z}_2) \backslash \widetilde \Omega^{(j)}(2^{j}\Theta \bot 2^{j+1}B_1),\] where $j=0$ or $1$ according as $m$ is even or odd.

\item[{\rm (2.2)}] Suppose that $m$ is even. Let $a \in {\bf Z}_2$ such that $a \equiv -1 \ {\rm mod} \ 4,$ and $B_1 \in S_{m-n_0-2}({\bf Z}_2)^{\times}.$  Then $\psi_{m-n_0-1,m}$ induces  a bijection  \begin{eqnarray*}
\lefteqn{GL_{m-n_0-1}({\bf Z}_2) \backslash \widetilde \Omega^{(1)}(a \bot 4B_1)} \\
&\simeq%
GL_{m}({\bf Z}_2) \backslash \widetilde \Omega^{(0)}(\Theta \bot  \smallmattwo(2;1;1;{1+a \over 2}) \bot  2B_1).
\end{eqnarray*}
\item[{\rm (2.3)}] Suppose that $m$ is even, and let  $B_1 \in S_{m-n_0-1}({\bf Z}_2)^{\times}.$  Then there exists a bijection  $\widetilde \psi_{m-n_0-1,m}$ 
\[GL_{m-n_0-1}({\bf Z}_2) \backslash \widetilde \Omega^{(1)}(4B_1)\simeq
GL_{m}({\bf Z}_2) \backslash \widetilde \Omega^{(0)}(\Theta \bot 2 \bot 2B_1).\]
\end{enumerate}
\hspace*{10mm}
{\rm (}As will be seen later, $\widetilde \psi_{m-n_0-1,m}$ is not induced from $\psi_{m-n_0-1,m}.${\rm )} \vspace*{2mm}

\noindent
{\rm (3)} The assertions {\rm (1),\,(2)} remain valid if one replaces $\widetilde \Omega^{(j)}(B)$ by $\widetilde \Omega^{(j)}(B,i).$ 
\end{lems}

\begin{proof} (1) Clearly 
$\psi_{m-n_0,m}$ induces an injection from 
$GL_{m-n_0}({\bf Z}_p) \backslash \widetilde \Omega^{(j)}(B_1)$ to $GL_{m}({\bf Z}_p) \backslash \widetilde \Omega^{(j)}(\Theta \bot B_1),$ which is denoted also by the same symbol $\psi_{m-n_0,m}.$ To prove the surjectivity  of $\psi_{m-n_0,m},$ take a representative $D$ of an element of $GL_{m}({\bf Z}_p) \backslash \widetilde \Omega^{(j)}(\Theta \bot B_1).$ Without loss of generality we may suppose that $D$ is a reduced matrix.
Since  $(\Theta \bot B_1)[D^{-1}]  \in  S_m({\bf Z}_p),$ we have 
$D=\mattwo(1_{n_0};0;0;D_1)$ with $D_1 \in \widetilde \Omega^{(j)}(B_1).$ This proves the assertion (1.1). The assertion (1.2) can be proved in the same way as above. 

(2) First we prove (2.1). As in (1), the mapping $\psi_{m-n_0,m}$ induces an injection from  $GL_{m-n_0}({\bf Z}_2) \backslash \widetilde \Omega^{(j)}(2^{j+1}B_1)$ to $GL_{m}({\bf Z}_2) \backslash \widetilde \Omega^{(j)}(2^{j}\Theta \bot 2^{j+1}B_1),$ which is denoted also by the same symbol $\psi_{m-n_0,m}.$ Then the surjectivity of $\psi_{m-n_0,m}$ in case $j=0$ can be proved in the same manner as (1). To prove the surjectivity of $\psi_{m-n_0,m}$ in case $j=1,$ take a reduced matrix  $D=\mattwo(D_1;D_{12};0;D_2)$ with $D_1 \in M_{n_0}({\bf Z}_2)^{\times}, D_2 \in M_{m-n_0}({\bf Z}_2)^{\times},D_{12} \in M_{n_0,m-n_0}({\bf Z}_2).$ If 
$(2\Theta \bot 4B_1)[D^{-1}] \in {\mathcal L}_{m,2}^{(1)},$ then there exists an element $(r_1,r_2) \in {\bf Z}_2^{n_0} \times {\bf Z}_2^{m-n_0}$ such that
$$2\Theta[D_1^{-1}] \equiv  -{}^tr_1r_1 \ {\rm mod} \ 4{\mathcal L}_{n_0,2}, $$
$$-2\Theta[D_1^{-1}]D_{12}D_2^{-1} \equiv  - {}^tr_2r_1 \ {\rm mod} \ 2M_{n_0,m-n_0}({\bf Z}_2),$$
and 
$$2\Theta[D_1^{-1}D_{12}D_2^{-1}]+4B_1[D_2^{-1}] \equiv  -{}^tr_2r_2 \ {\rm mod} \ 4{\mathcal L}_{m-n_0,2}.$$
We have $\nu(\det (2\Theta[D_1^{-1}])) \ge n_0,$ and $\nu(2\Theta)=n_0.$ Hence we have  $D_1=1_{n_0},$ and $r_1 \equiv 0 \ {\rm mod} \ 2.$ 
Hence  $4B_1[D_2^{-1}]  \in {\mathcal L}_{m-n_0}^{(1)},$ and $D_{12}D_2^{-1} \in M_{n_0,m-n_0}({\bf Z}_2).$ Hence $D=U\mattwo(1_{n_0};0;0;D_2)$ with $U \in GL_m({\bf Z}_p).$ Thus the surjectivity of $\psi_{m-n_0,m}$ can be proved in the same way as above. The assertion (2.2) can be proved in the same way as above.

To prove (2.3),  we may suppose that $n_0=0$ in view of (2.1). Let $D \in \widetilde \Omega^{(1)}(4B_1).$ Then  
$$4B_1[D^{-1}]= -\ ^tr_0 r_0 +4B'$$
with $r_0 \in {\bf Z}_2^{m-1}$ and $B' \in {\mathcal L}_{m-1,2}.$ Then we can take  $r \in {\bf Z}_2^{m-1}$ such that 
$$4 \ ^tD^{-1} \ ^tr r D^{-1} \equiv \ ^tr_0 r_0 \ {\rm mod} \ 4{\mathcal L}_{m-1,2}.$$
Furthermore, $2r D^{-1}$ is uniquely determined modulo $2{\bf Z}_2^{m-1}$ by $r_0.$ Put $\widetilde D=\mattwo(1;r;0;D).$ Then $\widetilde D$ belongs to $\widetilde \Omega^{(0)}(2 \bot 2B_1),$ and the mapping $D \mapsto \widetilde D$ induces a bijection in question.
\end{proof}

\begin{xcor}
Suppose that $m$ is even. Let $B \in {\mathcal L}_{m-1,p}^{(1)},$ and $$B^{(1)}=
\mattwo(1;r_B/2;{}^tr_B/2;{(B+{}^tr_B r_B)/4})$$ with $r_B \in {\bf Z}_p^{m-1}$ as defined in Section 3. 
Then there exists a bijection 
\[ \psi : GL_{m-1}({\bf Z}_p) \backslash \widetilde \Omega^{(1)}(B) \simeq 
GL_{m}({\bf Z}_p) \backslash \widetilde \Omega^{(0)}(2^{\delta_{2,p}}B^{(1)})
\] 
such that $\nu(\det (\psi(W)))=\nu(\det (W))$ for any $W \in GL_{m-1}({\bf Z}_p) \backslash \widetilde \Omega^{(1)}(B).$ Moreover, this induces a bijection $\psi_i$ from $ GL_{m-1}({\bf Z}_p) \backslash \widetilde \Omega^{(1)}(B,i)$ to \\
$GL_{m}({\bf Z}_p) \backslash \widetilde \Omega^{(0)}(2^{\delta_{2,p}}B^{(1)},i)$ for $i=0,\cdots, m-1.$ 
\end{xcor}

\begin{proof} Let $p \not=2.$ Then we may suppose $r_B=0,$ and the assertion follows from (1.2). Let $p=2.$ If $r_B \equiv 0 \ {\rm mod} \ 2$ we may suppose that $r_B=0,$ and the assertion follows from (2.3). If $r_B \not\equiv 0 \ {\rm mod} \ 2,$ we may suppose that $B=a \bot 4B_1$ with $B_1 \in {\mathcal L}_{m-2,2}^{\times}$ and $r_B=(1,0,\hdots,0).$ Thus the assertion follows from (2.2).
\end{proof}

\begin{lems}
Suppose that $p \not=2.$ \\
{\rm (1)}
 Let $B \in S_{m}({\bf Z}_p)^{\times}.$ Then 
$$\alpha_p(p^r d B)=p^{rm(m+1)/2}\alpha_p(B)$$
for any nonnegative integer $r$ and $d \in {\bf Z}_p^*.$ \\
\noindent
{\rm (2)}
  Let $U_1 \in GL_{n_0}({\bf Z}_p) \cap S_{n_0}({\bf Z}_p)$ and $B_1 \in S_{m-n_0}({\bf Z}_p)^{\times}.$ Then  
\begin{eqnarray*}
\lefteqn{
\alpha_p(pB_1 \bot U_1)= 2^{r(n_0)}\,\alpha_p(pB_1)
} \\
&\times& \hspace*{-2mm}
\left\{\begin{array}{ll}
{
\prod_{i=1}^{n_0/2} (1-p^{-2i})(1+\chi((-1)^{n_0/2} \det U_1)p^{-n_0/2})^{-1}}
&{ \ {\rm if} \ n_0 \ {\rm even}}, \\[2mm]
{
\prod_{i=1}^{(n_0-1)/2} (1-p^{-2i})}
& { \ {\rm if } \ n_0 \ {\rm odd}},
\end{array}\right.
\end{eqnarray*}
for $n_0 \ge 1,$ where $r(n_0)=0$ or $1$ according as $n_0=m$ or not. 
\end{lems}

\begin{proof} The assertion (1) follows from Theorem 5.6.4 (a) of \cite{Ki2}. 
 The assertion  (2)  follows from the formula on Page 110, line 4 from the bottom of  
Kitaoka \cite{Ki2}.  
\end{proof}


\begin{lems}
\noindent
{\rm (1)} Let $B \in S_{m}({\bf Z}_2)^{\times}.$ Then 
$$\alpha_2(2^r d B)=2^{rm(m+1)/2}\alpha_2(B)$$
for any nonnegative integer $r$ and $d \in {\bf Z}_2^*.$ 

\noindent
{\rm (2)} 
Let $n_0$ be even, and   $U_1 \in GL_{n_0}({\bf Z}_2) \cap S_{n_0}({\bf Z}_2)_e. $   
Then for $B_1 \in  S_{m-n_0}({\bf Z}_2)^{\times}$ we have 
\begin{eqnarray*}
\lefteqn{
\alpha_2(U_1 \bot 2B_1)= 2^{r(n_0)}\,\alpha_2(2B_1)
} \\
&\times& \hspace*{-2mm} 
\left\{ \begin{array}{ll}
{
\prod_{i=1}^{n_0/2} (1-2^{-2i})(1+\chi((-1)^{n_0/2} \det U_1)p^{-n_0/2})^{-1}}
&{} \\[2mm]
{}&
{\hspace*{-15mm}{\rm if} \ B_1 \in S_{m-n_0}({\bf Z}_2)_e}, \\[2mm]
{
\prod_{i=1}^{(n_0-1)/2} (1-2^{-2i})}&
{\hspace*{-15mm}{\rm if } \  B_1 \in S_{m-n_0}({\bf Z}_2)_o}
,
\end{array}
\right.
\end{eqnarray*}
  and for $u_0 \in {\bf Z}_2^*$ and $B_2 \in  S_{m-n_0-1}({\bf Z}_2)^{\times}$ we have 
  $$\alpha_2(u_0 \bot 2U_1 \bot 4B_2)= \alpha_2(2B_2)2^{m(m-1)/2+1}\prod_{i=1}^{n_0/2} (1-2^{-2i}).  $$
\noindent
{\rm (3)} Let  $u_0 \in {\bf Z}_2^*$ and $B_1 \in  S_{m-1}({\bf Z}_2)^{\times}.$ Then we have 
 $$\alpha_2(u_0 \bot  5B_1)= \alpha_2(u_0 \bot B_1).$$
\end{lems}

\begin{proof} 
The assertion (1) follows from Theorem 5.6.4 (a) of \cite{Ki2}. 
 The assertion  (2)  follows from (4) on Page 111 of \cite{Ki2}.  For a nondegenerate half-integral matrix $A,$ let $W_A$  be the  quadratic space over ${\bf Z}_p$ 
 associated with $A,$  and $n_{A,j}, q_{A,j}$ and $E_{A,j}$ be the quantities $n_j,q_j$ and $E_j,$ respectively, on Page 109 of \cite{Ki2} defined for $W_A.$ Then the transformation $u_0 \bot B_1 \mapsto u_0 \bot 5B_1$ does not change these quantities. This proves the assertion (3).


\end{proof}

Now let $R$ be a commutative ring. Then the group $GL_m(R) \times R^*$ acts on $S_m(R).$ We write $B_1 \approx_{R} B_2$ if $B_2 \sim_{R} \xi B_1$ with some $\xi \in R^*.$ 
Let $m$ be a positive integer. Then for $B \in S_m({\bf Z}_p)$ let $\widetilde {\mathcal S}_{m,p}(B)$ denote the set  of elements of $S_{m}({\bf Z}_p)$ such that $B' \approx_{{\bf Z}_p} B,$ and let ${\mathcal S}_{m-1,p}(B)$ denote the set  of elements of $S_{m-1}({\bf Z}_p)$ such that $1 \bot B' \approx_{{\bf Z}_p} B.$


\begin{lems}
Let $m$ be a positive even integer. Let $B \in S_m({\bf Z}_2)_o^{\times}.$ Then  
    $$\sum_{B' \in {\mathcal S}_{m-1,2}(B)/\sim} {1 \over \alpha_2(B')}={\#(\widetilde {\mathcal S}_{m,2}(B)/\!\!\sim)  \over 2 \alpha_2(B)}.$$
\end{lems}

 \begin{proof} For a positive integer $l$ let $l=l_1+\cdots +l_r$ be the partition of $l$ by positive integers, and $\{s_i \}_{i=1}^r $  the set of nonnegative integers such that $0 \le s_1 < \cdots <s_r.$ Then for a positive  integer $e$ let $S_l^0({\bf Z}_2/2^e{\bf Z}_2,\{l_i\},\{s_i\})$ be the subset of $S_l({\bf Z}_2/2^e{\bf Z}_2)$ consisting of symmetric matrices of the form $2^{s_1}U_1 \bot  2^{s_2}U_2 \bot \cdots \bot 2^{s_r} U_r$ with $U_i \in S_{l_i}({\bf Z}_2/2^e{\bf Z}_2)$ unimodular. Let $B \in S_m({\bf Z}_2)_o$ and $\det B=(-1)^{m/2}d.$ Then $B$ is equivalent, over ${\bf Z}_2,$ to a matrix of the following form:
 $$2^{t_1}W_1 \bot 2^{t_2}W_2 \bot \cdots \bot 2^{t_r} W_r,$$ 
where $0=t_1 < t_1<\cdots <t_r$ and $W_1,...,W_{r-1},$ and $W_r$ are unimodular matrices of degree $n_1,...,n_{r-1},$ and $n_r,$ respectively, and in particular, $W_1$ is odd unimodular. Then by Lemma 3.2 of \cite{I-S}, similarly to (3.5) of \cite{I-S}, for a sufficiently large integer $e,$ we have
\begin{eqnarray*}
\lefteqn{
{\#(\widetilde {\mathcal S}_{m,2}(B)/\!\!\sim)  \over  \alpha_2(B)}=\sum_{\widetilde B \in \widetilde {\mathcal S}_{m,2}(B)/\sim} {1 \over \alpha_2(\widetilde B)} } \\
&=& 2^{m-1}2^{ -\nu(d)+
\sum_{i=1}^r n_i(n_i-1)e/2-(r-1)(e-1)-
\sum_{1 \le j < i \le r}n_in_jt_j}  \\
&& \times \prod_{i=1}^r \#(SL_{n_i}({\bf Z}_2/2^e{\bf Z}_2))^{-1}\#\widetilde S^{(0)}_m({\bf Z}_2/2^e{\bf Z}_2,\{n_i\},\{t_i\},B),
\end{eqnarray*}
where $\widetilde S^{(0)}_m({\bf Z}_2/2^e{\bf Z}_2,\{n_i\},\{t_i\},B)$ is the subset of $S^{(0)}_m({\bf Z}_2/2^e{\bf Z}_2,\{n_i\},\{t_i\})$ consisting of matrices $A$ such that $A \approx_{{\bf Z}_2/2^e{\bf Z}_2} B.$ We note that our local density $\alpha_2(\widetilde B)$ is $2^{-m}$ times that in \cite{I-S} for $\widetilde B \in S_m({\bf Z}_2).$  If $n_1 \ge 2,$ put $r'=r,n_1'=n_1-1,n_2'=n_2,..,n_{r}'=n_r,$ and $t_i'=t_i$ for $i=1,...,r',$ and if $n_1=1,$ put $r'=r-1,n_i'=n_{i+1}$ and $t_i'=t_{i+1}$ for $i=1,...,r'.$  Let $S^{(0)}_{m-1}({\bf Z}_2/2^e{\bf Z}_2,\{n_i'\},\{t_i'\},B)$ be the subset of $S^{(0)}_{m-1}({\bf Z}_2/2^e{\bf Z}_2,\{n_i'\},\{t_i'\})$ consisting of matrices $B' \in S_{m-1}({\bf Z}_2/2^e{\bf Z}_2)$ such that $1 \bot B' \approx_{{\bf Z}_2/2^e{\bf Z}_2} B.$ Then, similarly, we obtain
\begin{eqnarray*}
\sum_{ B' \in  {\mathcal S}_{m-1,2}(B)/\sim} {1 \over \alpha_2(B')}  = 2^{m-2}2^{-\nu(d)+\sum_{i=1}^{r'} n_i'(n_i'-1)e/2-(r'-1)(e-1)-\sum_{1 \le j < i \le r'}n_i'n_j't_j'} \\
\times \prod_{i=1}^{r'} \#(SL_{n_i'}({\bf Z}_2/2^e{\bf Z}_2))^{-1}\#S^{(0)}_{m-1}({\bf Z}_2/2^e{\bf Z}_2,\{n_i'\},\{t_i'\},B).\hspace*{10mm}
\end{eqnarray*}
Take an element $A$ of $\widetilde S^{(0)}_m({\bf Z}_2/2^e{\bf Z}_2,\{n_i\},\{t_i\},B).$ Then  
$$A=2^{t_1}U_1 \bot 2^{t_2}U_2 \bot \cdots \bot 2^{t_r} U_r$$
 with $U_i \in S_{n_i}({\bf Z}_2/2^e{\bf Z}_2)$ unimodular. Put $U_1=(u_{\lambda \mu})_{n_1 \times n_1}.$ Then by the assumption there exists an integer $1 \le \lambda \le n_1$ such that $u_{\lambda\lambda} \in {\bf Z}_2^*.$ Let $\lambda_0$ be the least integer such that $u_{\lambda_0\lambda_0} \in {\bf Z}_2^*,$ and $V_1$ be the matrix obtained from $U_1$ by interchanging the first and $\lambda_0$-th lows and the first and $\lambda_0$-th columns. Write $V_1$ as $V_1=\mattwo(v_1;{\bf v}_1;{}^t{\bf v}_1;V')$ with $v_1 \in {\bf Z}_2^*,{\bf v}_1 \in M_{1,n_1-1}({\bf Z}_2),$ and $V' \in S_{n_1-1}({\bf Z}_2).$ Here we understand that $V'-\ {}^t{\bf v}_1{\bf v}_1$ is the empty matrix if $n_1=1.$ Then 
$$V_1 \sim \mattwo(v_1;0;0;V'-\ {}^t{\bf v}_1v_1^{-1}{\bf v}_1).$$
Then the map $A \mapsto {v_1}^{-1}(2^{t_1} (V'-\ {}^t{\bf v}_1v_1^{-1}{\bf v}_1) \bot 2^{t_2}U_2 \bot \cdots \bot 2^{t_r}U_r)$ induces a map $\Upsilon$ from $\widetilde S^{(0)}_m({\bf Z}_2/2^e{\bf Z}_2,\{n_i\},\{t_i\},B)$ to $S^{(0)}_{m-1}({\bf Z}_2/2^e{\bf Z}_2,\{n_i'\},\{t_i'\},B).$ By a simple calculation, we obtain 
$$\#\Upsilon^{-1}(B')=2^{(e-1)n_1}(2^{n_1}-1)$$
 for any $B' \in S^{(0)}_{m-1}({\bf Z}_2/2^e{\bf Z}_2,\{n_i'\},\{t_i'\},B).$ We also note that
$$\#SL_{n_1}({\bf Z}_2/2^e{\bf Z}_2)=2^{(e-1)(2n_1-1)}2^{n_1-1}(2^{n_1}-1)\#(SL_{n_1-1}({\bf Z}_2/2^e{\bf Z}_2)) \ {\rm or} \ 1$$
according as $n_1 \ge 2$ or $n_1=1,$ and 
\begin{eqnarray*}
\lefteqn{\sum_{i=1}^r n_i(n_i-1)e/2-(r-1)(e-1)-\sum_{1 \le j < i \le r}n_in_jt_j } \\
&=&e_{n_1}+\sum_{i=1}^{r'}n_i'(n_i'-1)e/2-(r'-1)(e-1)+\sum_{1 \le j < i \le r'}n_i'n_j't_j',
\end{eqnarray*}
where $e_{n_1}=(n_1-1)e$ or $e_{n_1}=1-e$ according as $n_1 \ge 2$ or $n_1=1.$   Hence 
\begin{eqnarray*}
\lefteqn{
2^{m-1}2^{-\nu(d)+\sum_{i=1}^r n_i(n_i-1)e/2-(r-1)(e-1)-\sum_{1 \le j  < i \le r}n_in_jt_j}
} \\
&& \times \prod_{i=1}^r \#(SL_{n_i}({\bf Z}_2/2^e{\bf Z}_2))^{-1}\#\widetilde S^{(0)}_m({\bf Z}_2/2^e{\bf Z}_2,\{n_i\},\{t_i\},B)
\end{eqnarray*}
\begin{eqnarray*}
\lefteqn{
\hspace*{2.5mm}=2 \cdot 2^{m-2}2^{-\nu(d)+\sum_{i=1}^{r'} n_i'(n_i'-1)e/2-(r'-1)(e-1)-\sum_{1 \le j \le i \le r'}n_i'n_j't_j'} }\\ 
&& \hspace*{2.5mm}\times \prod_{i=1}^{r '}\#(SL_{n_i'}({\bf Z}_2/2^e{\bf Z}_2))^{-1}\#S^{(0)}_{m-1}({\bf Z}_2/2^e{\bf Z}_2,\{n_i'\},\{t_i'\},B).
\end{eqnarray*}
This proves the assertion. 
 \end{proof} 
   





\subsection{Siegel series }

\noindent
{ }

\bigskip

For a half-integral matrix $B$ of degree $m$ over ${\bf Z}_p,$ let $(\overline{W},\overline {q})$ denote the quadratic space over ${\bf Z}_p/p{\bf Z}_p$ defined by the quadratic form $\overline {q}({\bf x})=B[{\bf x}] \ {\rm mod} \ p,$ and define the radical $R(\overline {W})$ of $\overline {W}$ by
$$R(\overline {W})=\{{\bf x} \in \overline {W} \, | \, \overline {B}({\bf x},{\bf y})=0 \ {\rm for \ any } \ {\bf y} \in \overline {W} \},$$
where $\overline {B}$ denotes the associated symmetric bilinear form of $\overline {q}.$ 
We then put $l_p(B)= {\rm rank}_{{\bf Z}_p/p{\bf Z}_p} R(\overline {W})^{\perp},$ where $R(\overline {W})^{\perp}$ is the orthogonal complement of $R(\overline {W})^{\perp}$ in $\overline {W}.$ Furthermore, in case $l_p(B)$ is even,  put $\overline {\xi}_p(B)=1$ or $-1$ according as $R(\overline {W})^{\perp}$ is hyperbolic or not. In case $l_p(B)$ is odd, we put $\overline {\xi}_p(B)=0.$ Here we make the convention that $\xi_p(B)=1$ if $l_p(B)=0.$ We note that $\overline {\xi}_p(B)$ is different from the $\xi_p(B)$ in general, but they coincide if $B \in {\mathcal L}_{m,p} \cap {1 \over 2}GL_m({\bf Z}_p).$  
 
Let $m$ be a positive even integer.
For $B \in {\mathcal L}_{m-1,p}^{(1)}$ put $$B^{(1)}=\mattwo(1;r/2; {}^tr/2;{(B+{}^trr)/4}),$$ where $r$ is an element of ${\bf Z}_p^{m-1}$ such that $B+{}^trr \in 4{\mathcal L}_{m-1,p}.$ Then we put
$\xi^{(1)}(B)=\xi(B^{(1)})$ and $\overline {\xi}^{(1)}(B)=\overline {\xi}(B^{(1)}).$ These do not depend on the choice of $r,$ and we have $\xi^{(1)}(B)=\chi((-1)^{m/2}\det B).$ 

Let $p \not=2.$ Let $j=0$ or $1.$ Then an element $B$ of ${\mathcal L}_{m-j,p}^{(j)}$ is equivalent, over ${\bf Z}_p$, to $\Theta \bot pB_1$
 with $\Theta \in GL_{m-n_1-j}({\bf Z}_p) \cap S_{m-n_1-j}({\bf Z}_p)$ and $B_1 \in S_{n_1}({\bf Z}_p)^{\times}.$ Then $\overline {\xi}^{(j)}(B)=0$ if $n_1$ is odd, and $\overline {\xi}^{(1)}(B)=\chi((-1)^{(m-n_1)/2} \det \Theta)$ if $n_1$ is even.  

Let $p=2.$ Then an element $B \in {\mathcal L}_{m-1,2}^{(1)}$ is equivalent, over ${\bf Z}_2,$  to a matrix of the form $2\Theta \bot B_1,$ where $\Theta \in GL_{m-n_1-2}({\bf Z}_2) \cap S_{m-n_1-2}({\bf Z}_2)_e$ and $B_1$ is one of the following three types: \begin{enumerate}
\item[(I)]
 $B_1=a  \bot 4B_2$ with $a \equiv -1 \ {\rm mod} \ 4, $ and $B_2 \in S_{n_1}({\bf Z}_2)_e^{\times}$; \vspace*{1mm}

\item[(II)]
 $B_1 \in  4S_{n_1+1}({\bf Z}_2)^{\times}$; \vspace*{1mm} 

\item[(III)]
 $B_1=a  \bot 4B_2$ with $a \equiv -1 \ {\rm mod} \ 4, $ and $B_2 \in S_{n_1}({\bf Z}_2)_o.$
 
 \end{enumerate}
Then $\overline {\xi}^{(1)}(B)=0$ if $B_1$ is of type (II) or (III). If $B_1$ is of type (I), then $(-1)^{(m-n_1)/2} a \det \Theta$  mod $({\bf Z}_2^*)^{\Box}$ is uniquely determined by $B,$  and we have $\overline {\xi}^{(1)}(B)=\chi((-1)^{(m-n_1)/2} a \det \Theta).$
Moreover, an element $B \in {\mathcal L}_{m,2}^{(0)}$ is equivalent, over ${\bf Z}_2,$  to a matrix of the form $\Theta \bot 2B_1,$ where $\Theta \in GL_{m-n_1-2}({\bf Z}_2) \cap S_{m-n_1-2}({\bf Z}_2)_e$ and $B_1 \in S_{n_1}({\bf Z}_2)^{\times} .$ 
Suppose that $p \neq 2,$ and let ${\mathcal U}={\mathcal U}_p$  be a complete set of representatives of ${\bf Z}_p^*/({\bf Z}_p^*)^{\Box}.$ Then, for each positive integer $l$ and $d \in {\mathcal U}_{p}$, there exists a unique, up to ${\bf Z}_p$-equivalence, element of $S_{l}({\bf Z}_p)  \cap GL_{l}({\bf Z}_p)$ whose determinant is $(-1)^{[(l+1)/2]}d,$ which will be denoted by  $\Theta_{l,d}.$ 
Suppose that $p=2,$ and put  ${\mathcal U}={\mathcal U}_{2}=\{1 ,5 \}.$ Then for each positive even integer $l$ and $d \in {\mathcal U}_{2}$ there exists a unique, up to ${\bf Z}_2$-equivalence, element of $S_{l}({\bf Z}_2)_{e} \cap GL_{l}({\bf Z}_2)$ whose determinant is  $(-1)^{l/2} d,$ which will be also denoted by  $\Theta_{l,d}.$ In particular, if $p$ is any prime number and $l$ is even, we put $\Theta_l=\Theta_{l,1}$ We make the convention that $\Theta_{l,d}$ is the empty matrix if $l=0.$ For an element $d \in {\mathcal U}$ we use the same symbol $d$ to denote the coset $d$ mod  $({\bf Z}_p^*)^{\Box}.$

 For $T \in {\mathcal L}_{n-1,p}^{(1)},$  let  $\widetilde F_p^{(1)}(T,X)$ be the polynomial in $X$ and $X^{-1}$
defined in Section 3.  
We also define a polynomial $G_p^{(1)}(T,X)$ in $X$ by
\begin{eqnarray*}
\lefteqn{
G_p^{(1)}(T,X)} \\
&=&\sum_{i=0}^{n-1} (-1)^i p^{i(i-1)/2} (X^2p^{n})^i \sum_{D \in GL_{n-1}({\bf Z}_p) \backslash {\mathcal D}_{n-1,i}} F_p^{(1)}(T[D^{-1}],X). 
\end{eqnarray*}

\begin{lems} 
 Let $n$ be the fixed positive even integer. Let $B \in {\mathcal L}_{n-1,p}^{(1)}$ and put  $\xi_0=\chi((-1)^{n/2} \det B).$ \\
 {\rm (1)} Let $p \not=2, $ and suppose that  $B=\Theta_{n-n_1-1,d} \bot pB_1$  with $d \in {\mathcal U}$ and $B_1 \in S_{n_1}({\bf Z}_p)^{\times}.$ 
Then 
\begin{eqnarray*}
\lefteqn{
G_p^{(1)}(B,Y)
}{ } \\
&=& \hspace*{-2.5mm}
\left\{
\begin{array}{cl}
\displaystyle { 1-\xi_0 p^{n/2}Y  \over 1-p^{n_1/2+n/2}\overline{\xi}^{(1)}(B)Y }\prod_{i=1}^{n_1/2}(1-p^{2i+n}Y^2)
& {\rm if} \ n_1 \ {\rm is \ even,}  \\[2mm]
\displaystyle (1-\xi_0 p^{n/2}Y)\prod_{i=1}^{(n_1-1)/2}(1-p^{2i+n}Y^2) 
& {\rm if} \  n_1 \ {\rm is \ odd}
.
\end{array}\right. 
\end{eqnarray*}
 {\rm (2)} Let $p=2.$ Suppose that $n_1$ is even and that $B=2\Theta \bot B_1 $  with $\Theta \in S_{n-n_1-2}({\bf Z}_2)_e \cap  GL_{n-n_1-2}({\bf Z}_2) $ and $B_1 \in S_{n_1+1}({\bf Z}_2)^{\times}.$  Then 
\begin{eqnarray*}
\lefteqn{G_2^{(1)}(B,Y)}{} \\
&=& \hspace*{-2.5mm}
\left\{ \begin{array}{cl}
{\displaystyle {1-\xi_02^{n/2}Y \over  1-2^{n_1/2+n/2} \overline{\xi}^{(1)}(B)Y}\prod_{i=1}^{n_1/2}(1-2^{2i+n}Y^2)}
&{{\rm if} \  B_1 \ {\rm is \ of \ type \ (I),}}\\[2.5mm]
\displaystyle (1-\xi_02^{n/2}Y){\prod_{i=1}^{n_1/2}(1-2^{2i+n}Y^2)} 
&\hspace*{-5mm}{\rm if} \ { B_1 \ {\rm is \ of \ type \ (II) \ or \ (III)}}. 
 
 \end{array}\right.
\end{eqnarray*} 
\end{lems}

\begin{proof} By Corollary to Lemma 4.1.2 and by definition we have $G_p^{(1)}(B,Y)=G_p(B^{(1)},Y).$ Thus the assertion follows from Lemma 9 of \cite{Ki1}. 
\end{proof}

\noindent
{\bf Remark.} Throughout the above lemma, $\overline{\xi}^{(1)}(B)=\xi_0(B)$ if $n_1=0.$ Hence we have $G_p^{(1)}(B,Y)=1$ in this case.

\begin{lems}
Let $B \in {\mathcal L}_{n-1,p}^{(1)}.$ Then 
$$\widetilde F_p^{(1)}(B,X)=\sum_{B' \in  {\mathcal L}_{n-1,p}^{(1)}/ GL_{n-1}({\bf Z}_p)  } X^{-\mathfrak{e}^{(1)}(B')}{\alpha_p(B',B) \over \alpha_p(B')}$$
$$ \times  G_p^{(1)}(B',p^{(-n-1)/2}X)(p^{-1}X)^{(\nu(\det B)-\nu(\det B'))/2}.$$
\end{lems}

\begin{proof} We have 
\begin{eqnarray*}
\lefteqn{
\widetilde F_p^{(1)}(B,X) }\\
&=&\sum_{W \in  GL_{n-1}({\bf Z}_p) \backslash \widetilde \Omega^{(1)}(B)} X^{-\mathfrak{e}^{(1)}(B)}G_p^{(1)}(B[W^{-1}],p^{(-n-1)/2}X)X^{2\nu(\det W)} \\
&=& \sum_{B' \in   {\mathcal L}_{n-1,p}^{(1)}/GL_{n-1}({\bf Z}_p)}
\sum_{W \in  GL_{n-1}({\bf Z}_p) \backslash \widetilde \Omega(B',B)} 
\hspace*{-2.5mm}X^{-\mathfrak{e}^{(1)}(B)}G_p^{(1)}(B',p^{(-n-1)/2}X)X^{2\nu(\det W)} \\
&=&\sum_{B' \in   {\mathcal L}_{n-1,p}^{(1)}/GL_{n-1}({\bf Z}_p)  }X^{-\mathfrak{e}^{(1)}(B')}  \#(GL_{n-1}({\bf Z}_p) \backslash \widetilde \Omega(B',B))p^{(\nu(\det B)-\nu(\det B'))/2} \\
&& \hspace*{5mm}\times G_p^{(1)}(B',p^{(-n-1)/2}X)(p^{-1}X)^{(\nu(\det B)-\nu(\det B'))/2}.
\end{eqnarray*}
Thus the assertion follows from (2) of Lemma 4.1.1.
\end{proof}

\bigskip

 \subsection{Certain reduction formulas}
 
 \noindent
{ }

\bigskip

To give an explicit formula for the power series $P_{n-1}^{(1)}(d_0,\omega,X,t),$ we give certain reduction formulas, by which we can express  $P_{n-1}^{(1)}(d_0,\omega,X,t)$ in terms of the power series defined in \cite {I-S}. 
  First we  review the notion of canonical forms of the quadratic forms over ${\bf Z}_2$ in the sense of Watson \cite{W}. Let $B \in  {\mathcal L}_{m,2}^{\times}.$ 
Then $B$ is equivalent, over ${\bf Z}_2,$ to a  matrix of the following form:
$$\bot_{i=0}^r 2^{i}(V_i \bot U_i),$$
where $V_i=\bot_{j=1}^{k_i} c_{ij}$ with $0 \le k_i \le 2, c_{ij} \in {\bf Z}_2^*$ and $U_i= {1 \over 2}\Theta_{m_i,d}$ with $0 \le m_i, d \in {\mathcal U}.$  The degrees $k_i$ and $m_i$ of the matrices are uniquely determined by $B.$ Furthermore we can take the matrix  $\bot_{i=0}^r 2^{i}(V_i \bot U_i)$  uniquely so that it satisfies the following conditions:
\begin{enumerate}
\item[{\rm (c.1)}] $c_{i1}=\pm 1$ or $\pm 3$ if $k_i=1$ and $(c_{i1},c_{i2})=(1,\pm 1),(1,\pm 3),(-1,-1)$, or $(-1,3)$ if $k_i=2;$

\item[{\rm (c.2)}] $k_{i+2}=k_i=0$ if $U_{i+2}={1 \over 2}\Theta_{m_{i+2},5}$ with $m_{i+2} >0;$  

\item[{\rm (c.3)}] $-\det V_i \equiv 1 \ {\rm  mod} \ 4$ if $k_i=2$ and $U_{i+1}={1 \over 2}\Theta_{m_{i+1},5}$ with $m_{i+1} > 0;$

\item[{\rm (c.4)}] $(-1)^{k_i-1} \det V_i \equiv 1 \ {\rm  mod} \  4$ if $k_i,k_{i+1} > 0;$ 

\item[{\rm (c.5)}] $V_i \not= \mattwo(-1;0;0;c_{i2})$ if $k_{i-1}>0;$

\item[{\rm (c.6)}] $V_i=\phi, (\pm 1),\mattwo(1;0;0;\pm 1),$ or $\mattwo(-1;0;0;-1)$ if $k_{i+2}>0.$

\end{enumerate}

\noindent 
The matrix satisfying the conditions (c.1),\, $\cdots$,\, (c.6) is called the {\it canonical form} of $B,$ and denote it by $C(B).$ Now for  $V =\bot_{j=1}^k c_j$ with $1 \le k \le 2,$ put $\widetilde V=5c_1$ or $\widetilde V=5c_1 \bot c_2$ according as $V=c_1$ or $V=c_1 \bot c_2.$
 
\begin{lems}
For $B \in S_m({\bf Z}_2)_o^{\times},$ let $C(B) =V_0 \bot \bot_{i=1}^r (U_i \bot V_i)$ be the  canonical form of $B$ stated as above.  Let $l=l_B$ be the smallest integer such that $k_{2l+2}=0.$ Then we have
$$C(\widetilde V_0 \bot \bot_{i=1}^r (U_i \bot V_i))= V_0 \bot \bot_{i=1}^{2l-1} (U_i \bot V_i) \bot U_{2l} \bot C(\widetilde V_{2l}) \bot \bot_{i=2l+1}^r(U_i \bot V_i). $$ 
\end{lems}
\begin{proof}
We note that $5a_1 \bot 4a_2 \sim a_1 \bot 4 \cdot 5a_2$ for $a_1,a_2 \in {\bf Z}_2^*.$ Hence we have
$$\widetilde V_0  \bot \bot_{i=1}^l  V_{2i} \sim V_0 \bot \bot_{i=1}^{l-1}  V_{2i} \bot \widetilde V_{2l}.$$
This proves the assertion.
\end{proof}
\begin{xcor}
For $B,B' \in S_{2m+1}(d_0)_o,$ let $C(B)=V_0 \bot \bot_{i=1}^r (U_i \bot V_i)$ and $C(B')=V_0' \bot \bot_{i=1}^{r'} (U_i' \bot V_i')$
with $V_0 =\bot_{j=1}^{k_0} c_{0j}$ and $V_0' =\bot_{j=1}^{k_0} c_{0j}'.$
Put 
\begin{center}
$B_1= \bot_{j=2}^{k_0} c_{0j}  \bot_{i=1}^r (U_i \bot V_i)$ and $B_1'= \bot_{j=2}^{k_0'} c_{0j}'  \bot_{i=1}^{r'} (U_i' \bot V_i').$ \end{center}
Then $B \sim B'$ if and only if $c_{01} \bot 5B_1 \sim c_{01}' \bot 5B_1'.$
\end{xcor}
\begin{proof} We note that $c_{01} \bot 5B_1 \sim c_{01}' \bot 5B_1'$ if and only if $5c_{01} \bot B_1 \sim 5c_{01}' \bot B_1'.$
Hence the assertion follows from the lemma.
\end{proof}


   
The following lemma follows from Theorem 3.4.2 of \cite{Ki2}.
   
\begin{lems}  
  Let $m$ and $r$ be integers such that $0 \le r \le m,$   and  $d_0 \in {\bf Z}_p^{\times}.$ 
  
  \noindent
  {\rm (1)} Let $p\not=2,$ and $T \in S_r({\bf Z}_p,d_0).$  Then for any $d \in {\mathcal U}$ we have 
  $$\varepsilon(\Theta_{m-r,d} \bot T)=((-1)^{[(m-r+1)/2]}d,d_0)_p\varepsilon(T).$$
  Furthermore we have  
    \begin{eqnarray*}
\varepsilon(pT)=
\left\{\begin{array}{ll}
{(p,d_0)_p\varepsilon(T)} &{ \ {\rm if} \ r \ {\rm even}}, \\[2mm]
{(p,(-1)^{(r+1)/2})_p\varepsilon(T)} & { \ {\rm if } \ r \ {\rm odd}},
\end{array}\right.
\end{eqnarray*}
 and $\varepsilon(aT)=(a,d_0)_p^{r+1}\varepsilon(T)$ for any $a \in {\bf Z}_p^*.$ 
  
 \noindent
 {\rm (2)} Let $p=2,$ and $T \in S_r({\bf Z}_2,d_0).$  Suppose that $m-r$ is even, and let $d \in {\mathcal U}.$  Then for $\Theta=2 \Theta_{m-r,d}$ or $
 2 \Theta_{m-r-2} \bot (-d)$, we have 
 $$\varepsilon(\Theta \bot T)=(-1)^{(m-r)(m-r+2)/8}((-1)^{(m-r)/2}d,(-1)^{[(r+1)/2]}d_0)_2\varepsilon(T)$$
and 
$$\varepsilon(\Theta_{m-r,d} \bot T)=(-1)^{(m-r)(m-r+2)/8}(2,d)_2((-1)^{(m-r)/2}d,(-1)^{[(r+1)/2]}d_0)_2\epsilon(T).$$
 Furthermore we have 
 $\varepsilon(2T)=(2,d_0)_2^{r+1}\varepsilon(T),$ 
$$\varepsilon(a \bot T)=(a,(-1)^{[(r+1)/2]+1}d_0)_2 \varepsilon(T)$$
 for any $a \in {\bf Z}_2^*,$
 and 
 \begin{eqnarray*}
 \varepsilon(aT)=
\left\{\begin{array}{ll}
{(a,d_0)_2\varepsilon(T)}  & { \ {\rm if} \ r \ {\rm even}}, \\[2mm]
{(a,(-1)^{(r+1)/2})_2\varepsilon(T)} & { \ {\rm if } \ r \ {\rm odd}}
\end{array}\right.
\end{eqnarray*}
for any $a \in {\bf Z}_2^*.$ 
 \end{lems}
 

Henceforth,  we sometimes abbreviate $S_r({\bf Z}_p)$ and $S_r({\bf Z}_p,d)$ as $S_{r,p}$ and $S_{r,p}(d),$ respectively.
Furthermore we abbreviate $S_r({\bf Z}_2)_x$ and $S_r({\bf Z}_2,d)_x$ as $S_{r,2;x}$ and $S_{r,2}(d)_x,$ respectively, for $x=e,o.$  Let $R$ be a commutative ring. A function $H$ defined on a subset ${\mathcal S}$ of $S_m({\bf Q}_p)$ with values in $R$ is said to be {\it $GL_m({\bf Z}_p)$-invariant} if 
$H(A[U])=H(A)$ for any $U \in GL_m({\bf Z}_p)$ and $A \in {\mathcal S}.$  Let $p \not=2.$  Let  $\{H_{2r+j,\xi}^{(j)} \ |  \ j \in \{0,1\},
1-j \le r \le n/2-j , \xi=\pm 1 \}$  be a set of  $GL_{2r+j}({\bf Z}_p)$-invariant functions on $S_{2r+j}({\bf Z}_p)^{\times}$ with values in $R$ satisfying the following conditions for any positive even integer $m \le n$:
\begin{enumerate}
\item[{\rm (H-}$p${\rm -0)}] $H_{m,\xi}^{(0)}(\Theta_{m,d})=1$ and $H_{m-1,\xi}^{(1)}(\Theta_{m-1,d})=1$ for $d \in {\mathcal U};$
\item[{\rm (H-}$p${\rm -1)}] $H_{m,\xi}^{(0)}(\Theta_{m-2r,d} \bot pB)=H_{2r,\xi\chi(d)}^{(0)}(pB)$  for any $r \le m/2-1, 
\xi=\pm 1, d \in {\mathcal U}$ and $B \in S_{2r}({\bf Z}_p)^{\times};$ 
\item[{\rm (H-}$p${\rm -2)}] $H_{m-1,\xi}^{(1)}(\Theta_{m-2r-2,d} \bot pB)=H_{2r+1,\xi}^{(1)}(pdB)$    for any 
$r \le m/2-2, \xi=\pm 1, d \in {\mathcal U}$ and $B \in S_{2r+1}({\bf Z}_p)^{\times};$
\item[{\rm (H-}$p${\rm -3)}] $H_{m,\xi}^{(0)}(\Theta_{m-2r-1,d} \bot pB)=H_{2r+1,\xi}^{(1)}(-pdB)$  for any $r \le m/2-2, 
\xi=\pm 1,$ and $B \in S_{2r+1}({\bf Z}_p)^{\times};$ 
\item[{\rm (H-}$p${\rm -4)}] $H_{m-1,\xi}^{(1)}(\Theta_{m-2r-1,d}  \bot pB)=H_{2r,\xi \chi(d)}^{(0)}(pB)$ for any 
$r \le m/2 -2, \xi=\pm 1, d \in {\mathcal U}$ and $B \in S_{2r}({\bf Z}_p)^{\times};$ 
\item[{\rm (H-}$p${\rm -5)}] $H_{2r,\xi}^{(0)}(dB)=H_{2r,\xi}^{(0)}(B)$ for  any $r \le m/2, \xi=\pm 1, 
d \in {\bf Z}_p^*$ and $B \in S_{2r}({\bf Z}_p)^{\times}.$   \vspace*{2mm}
\end{enumerate} 
Let $d_0 \in {\mathcal F}_p,$  and $m$ be a positive even integer such that $m \le n.$ Then for each $0 \le r \le m/2 -1 $ we put
\begin{eqnarray*}
\lefteqn{Q^{(1)}(d_0,H_{m-1,\xi}^{(1)},2r+1,\epsilon^l,t)=\kappa(d_0,m-1,l)^{-1}} \\
&\times& \sum_{d \in {\mathcal U}} \sum_{B \in p^{-1}S_{2r+1,p}(d_0d) \cap S_{2r+1,p} }\hspace*{-3.5mm} {H_{m-1,\xi}^{(1)}(\Theta_{m-2r-2,d} \bot pB) \epsilon(\Theta_{m-2r-2,d}  \bot pB)^l \over \alpha_p(\Theta_{m-2r-2,d} \bot pB)}t^{\nu(\det (pB))}.
\end{eqnarray*}
Let $d \in {\mathcal U}.$ Then we put 
\begin{eqnarray*}
\lefteqn{Q^{(1)}(d_0,d,H_{m-1,\xi}^{(1)},2r,\epsilon^l,t)=\kappa(d_0,m-1,l)^{-1} } \\
&\times& \sum_{B \in S_{2r,p}(d_0d)} {H_{m-1,\xi}^{(1)}(\Theta_{m-2r-1,d} \bot pB) \epsilon(\Theta_{m-2r-1,d} \bot pB)^l \over \alpha_p(\Theta_{m-2r-1,d} \bot pB )}t^{\nu(\det (pB))}
\end{eqnarray*}
for each $1 \le r \le m/2-1,$ and  
\begin{eqnarray*}
\lefteqn{Q^{(0)}(d_0,d,H_{m,\xi}^{(0)},2r,\epsilon^l,t)} \\
&=& \sum_{B \in S_{2r,p}(d_0d)} {H_{m,\xi}^{(0)}(\Theta_{m-2r,d} \bot pB) \epsilon(\Theta_{m-2r,d}
  \bot pB)^l \over \alpha_p(\Theta_{m-2r,d} \bot pB)}t^{\nu(\det (pB))} \end{eqnarray*}
for each $1 \le r \le m/2.$   
Here we make the convention that 
$$Q^{(0)}(d_0,1,H_{m,\xi}^{(0)},m,\epsilon^l,t)=\sum_{B \in S_{m,p}(d_0)} {H_{m,\xi}^{(0)}(pB) \epsilon(pB)^l \over \alpha_p(pB)}t^{\nu(\det (pB))}. $$
We also define
$$Q^{(1)}(d_0,d,H_{m-1,\xi}^{(1)},0,\epsilon^l,t)=Q^{(0)}(d_0,d,H_{m,\xi}^{(0)},0,\epsilon^l,t)=\delta(d,d_0),$$
where $\delta(d,d_0)=1$ or $0$ according as $d=d_0$ or not. Furthermore put 
\begin{eqnarray*}
\lefteqn{Q^{(0)}(d_0,H_{m,\xi}^{(0)},2r+1,\epsilon^l,t)} \\
&=& \sum_{d \in {\mathcal U}} \sum_{B \in p^{-1}S_{2r+1,p}(d_0d) \cap S_{2r+1,p}} 
\hspace*{-5mm}
{H_{m,\xi}^{(0)}(-\Theta_{m-2r-1,d} \bot pB) \epsilon(-\Theta_{m-2r-1,d}  \bot pB)^l \over \alpha_p(-\Theta_{m-2r-1,d} \bot pB)}t^{\nu(\det (pB))}
\end{eqnarray*}
for each $0 \le r \le m/2-1.$

Let  $\{H_{2r+j,\xi}^{(j)} \ | \ j \in \{0,1\},
1-j \le r \le n/2-j , \xi=\pm 1 \}$  be a set of  $GL_{2r+j}({\bf Z}_2)$-invariant functions on $S_{2r+j}({\bf Z}_2)^{\times}$ with values in $R$ satisfying the following conditions for any positive even integer $m \le n$:
\begin{enumerate}
\item[{\rm (H-2-0)}] $H_{m,\xi}^{(0)}(\Theta_{m,d})=H_{m-1,\xi}^{(1)}(-d \bot 2\Theta_{m-2})=1$ for $d \in {\mathcal U};$
\item[{\rm (H-2-1)}] $H_{m,\xi}^{(0)}(\Theta_{m-2r,d} \bot 2B)=H_{2r,\xi\chi(d)}^{(0)}(2B)$ for any $r \le m/2-1, \xi=\pm 1, d \in {\mathcal U}$ and $B \in S_{2r}({\bf Z}_2)^{\times};$ 
\item[{\rm (H-2-2)}] $H_{m-1,\xi}^{(1)}( 2\Theta_{m-2r-2,d} \bot 4B)=H_{2r+1,\xi}^{(1)}(4dB)$    for any $r \le m/2 -2,\xi=\pm 1, d \in {\mathcal U}$ and $B \in S_{2r+1}({\bf Z}_2)^{\times};$ 
\item[{\rm (H-2-3)}] $H_{m,\xi}^{(0)}(2 \bot \Theta_{m-2r-2} \bot 2B)=H_{2r+1,\xi}^{(1)}(4B)$  for any $r \le m/2 -2, \xi=\pm 1,$ and $B \in S_{2r+1}({\bf Z}_2)^{\times};$ 
\item[{\rm (H-2-4)}] $H_{m-1,\xi}^{(1)}( -a \bot 2\Theta_{m-2r-2} \bot 4B)=H_{2r,\xi \chi(a)}^{(0)}(2B)$ for any $r \le m/2 -2, \xi=\pm 1, a \in {\mathcal U}$ and $B \in S_{2r}({\bf Z}_2)^{\times};$ 
\item[{\rm (H-2-5)}] $H_{2r,\xi}^{(0)}(dB)=H_{2r,\xi}^{(0)}(B)$ for any  $r \le m/2, \xi=\pm 1, d \in {\bf Z}_2^*$ and $B \in S_{2r}({\bf Z}_2)^{\times};$ 
\item[{\rm (H-2-6)}] $H_{2r+1,\xi}^{(1)}(4(u_0 \bot  B))=H_{2r+1,\xi}^{(1)}(4(u_0 \bot  5B))$ for any $r \le m/2 -1,\xi=\pm 1$ and $u_0 \in {\bf Z}_2^*, B \in S_{2r}({\bf Z}_2)^{\times}.$
 \vspace*{2mm}
\end{enumerate}
 Let $d_0 \in {\mathcal F}_2,$ and $m$ be a positive even integer such that $m \le n.$ Then for each $0 \le r \le m/2 -1,$ we put 
\begin{eqnarray*}
\lefteqn{Q^{(11)}(d_0,H_{m-1,\xi}^{(1)},2r+1,\varepsilon^l,t)=\kappa(d_0,m-1,l)^{-1} t^{2-m}}\\
&\times&  \sum_{d \in {\mathcal U}} \sum_{B \in S_{2r+1,2}(d_0d)_e} {H_{m-1,\xi}^{(1)}(2\Theta_{m-2r-2,d} \bot 4B) \varepsilon^l(2\Theta_{m-2r-2,d}  \bot 4B) \over \alpha_2(2\Theta_{m-2r-2,d} \bot 4B)} \\
&& \hspace*{25mm}\times t^{m-2r-2+\nu(\det (4B))}, 
\end{eqnarray*}
\begin{eqnarray*}
\lefteqn{Q^{(12)}(d_0,H_{m-1,\xi}^{(1)},2r+1,\varepsilon^l,t)=\kappa(d_0,m-1,l)^{-1}t^{2-m} } \\
&\times& \sum_{B \in S_{2r+1,2}(d_0)_o } {H_{m-1,\xi}^{(1)}(2\Theta_{m-2r-2} \bot 4B) \varepsilon^l(2\Theta_{m-2r-2}  \bot 4B) \over \alpha_2(2\Theta_{m-2r-2} \bot 4B)} \\
&&\hspace*{25mm}\times t^{m-2r-2+\nu(\det (4B))}, 
\end{eqnarray*}
  and 
\begin{eqnarray*}
\lefteqn{Q^{(13)}(d_0,H_{m-1,\xi}^{(1)},2r+1,\varepsilon^l,t)=\kappa(d_0,m-1,l)^{-1}t^{2-m}} \\
&& \times  \sum_{B \in S_{2r+2,2}(d_0)_o} H_{m-1,\xi}^{(1)}(-1 \bot 2\Theta_{m-2r-4} \bot 4B) \\
&& \times { \varepsilon^l(-1 \bot 2\Theta_{m-2r-4}
  \bot 4B) \over \alpha_2(-1 \bot 2\Theta_{m-2r-4} \bot 4B)}\,\,t^{m-2r-4+\nu(\det (4B))}.
  \end{eqnarray*} 
 Moreover put    
\begin{eqnarray*}
\lefteqn{Q^{(1)}(d_0,H_{m-1,\xi}^{(1)},2r+1,\varepsilon^l,t)=Q^{(11)}(d_0,H_{m-1,\xi}^{(1)},2r+1,\varepsilon^l,t)} \\
&+&Q^{(12)}(d_0,H_{m-1,\xi}^{(1)},2r+1,\varepsilon^l,t)+Q^{(13)}(d_0,H_{m-1,\xi}^{(1)},2r+1,\varepsilon^l,t).\end{eqnarray*}
  We note that  
\begin{eqnarray*}
\lefteqn{ Q^{(1)}(d_0,H_{m-1,\xi}^{(1)},m-1,\epsilon^l,t)=\kappa(d_0,m-1,l)^{-1}t^{2-m}} \\
&\times& \sum_{B \in S_{m-1,2}(d_0)}H_{m-1,\xi}^{(1)}(4B){ \epsilon(4B)^l \over \alpha_2(4B)} t^{\nu(\det (4B))}. 
\end{eqnarray*}
 Let $d \in {\mathcal U}.$ Then we put 
  \begin{eqnarray*}
\lefteqn{ Q^{(1)}(d_0,d,H_{m-1,\xi}^{(1)},2r,\epsilon^l,t)=\kappa(d_0,m-1,l)^{-1} t^{2-m}} \\
&\times& \sum_{B \in S_{2r,2}(d_0d)_e }H_{m-1,\xi}^{(1)}(-d \bot 2\Theta_{m-2r-2} \bot 4B) \\
 & \times&  {\epsilon(-d \bot 2\Theta_{m-2r-2}
  \bot 4B)^l \over \alpha_2(-d \bot 2\Theta_{m-2r-2} \bot 4B)}t^{m-2r-2+\nu(\det (4B))}, 
  \end{eqnarray*}
  for each $1 \le r \le m/2 -1,$ and  
\begin{eqnarray*}
\lefteqn{ Q^{(0)}(d_0,d,H_{m,\xi}^{(0)},2r,\epsilon^l,t)=\kappa(d_0,m,l)^{-1} } \\
&\times& \sum_{B \in S_{2r,2}(d_0d)_e} {H_{m,\xi}^{(0)}(\Theta_{m-2r,d} \bot 2B) \epsilon(\Theta_{m-2r,d}
  \bot 2B)^l \over \alpha_2(\Theta_{m-2r,d} \bot 2B)}t^{\nu(\det (2B))} 
  \end{eqnarray*}
for each $1 \le r \le m/2,$ where $ \kappa(d_0,m,l)=\{(-1)^{m(m+2)/8}((-1)^{m/2}2,d_0)_2\}^{l}.$ Here we make the convention that 
$$ Q^{(0)}(d_0,1,H_{m,\xi}^{(0)},m,\epsilon^l)=\kappa(d_0,m,l)^{-1} \sum_{B \in S_{m,2}(d_0)_e } {H_{m,\xi}^{(0)}(2B) \epsilon(2B)^l \over \alpha_2(2B)}t^{\nu(\det (2B))}. $$
 We also define 
 $$Q^{(1)}(d_0,d,H_{m-1,\xi}^{(1)},0,\epsilon^l,t)=Q^{(0)}(d_0,d,H_{m,\xi}^{(0)},0,\epsilon^l,t)=\delta(d,d_0).$$
 Furthermore put 
\begin{eqnarray*}
\lefteqn{ Q^{(0)}(d_0,H_{m,\xi}^{(0)},2r+1,\epsilon^l,t)= \kappa(d_0,m,l)^{-1} } \\
&\times& \sum_{B \in S_{2r+2,2}(d_0)_o} {H_{m,\xi}^{(0)}(\Theta_{m-2r-2} \bot 2B) \epsilon(\Theta_{m-2r-2}
  \bot 2B)^l \over \alpha_2(\Theta_{m-2r-2} \bot 2B)}t^{\nu(\det (2B))}
  \end{eqnarray*}
  for $0 \le r \le m/2 -1.$ Henceforth, for $d_0 \in   {\mathcal F} _ p$ and nonnegative integers $m,r$ such that $r \le m,$ put
 ${\mathcal U}(m,r,d_0)=\{1\},{\mathcal U} \cap \{d_0\},$ or ${\mathcal U}$ according as $r=0, \ r=m \ge 1,$ or $1 \le  r \le m-1.$

\begin{props}   
  
Let the notation be as above. 
\begin{enumerate}
\item[{\rm (1)}] For $0 \le r \le (m-2)/2,$ we have  $$Q^{(0)}(d_0,H_{m,\xi}^{(0)},2r+1,\varepsilon^l,t)={Q^{(1)}(d_0,H_{2r+1,\xi }^{(1)},2r+1,\varepsilon^l,t) \over \phi_{(m-2r-2)/2}(p^{-2})}$$
if $l\nu(d_0)=0,$ and  
$$Q^{(0)}(d_0,H_{m,\xi}^{(0)},2r+1,\varepsilon,t)=0$$
if $\nu(d_0)>0.$   

\item[{\rm (2)}] For $1 \le r \le m/2$ and $d \in {\mathcal U}(m,m-2r,d_0),$ we have   
\begin{eqnarray*}
\lefteqn{Q^{(0)}(d_0,d,H_{m,\xi}^{(0)},2r,\varepsilon^l,t)} \\
&=&{(1+p^{-(m-2r)/2}\chi(d))  Q^{(0)}(d_0d,1,H_{2r,\xi \chi(d)}^{(0)},2r,\varepsilon^l,t) \over 2 \phi_{(m-2r)/2}(p^{-2})}
\end{eqnarray*}
  if $l\nu(d_0)=0,$ and  
  $$Q^{(0)}(d_0,d,H_{m,\xi}^{(0)},2r,\varepsilon,t)=0$$
   if $\nu(d_0) >0.$    
\end{enumerate}
  
\end{props}
   
  \begin{proof}  First suppose that $p \not=2.$ We note that  
   $$(-\Theta_{m-2r-1,d}) \bot pB \sim d (-\Theta_{m-2r-1}) \bot pB \approx (-\Theta_{m-2r-1}) \bot dpB$$
   for $d \in {\mathcal U}$ and $B \in p^{-1}S_{2r+1,p}(d_0d),$ and  the mapping    $$p^{-1}S_{2r+1,p}(d_0d)  \cap S_{2r+1,p} \ni B \mapsto  dB \in p^{-1}S_{2r+1,p}(d_0) \cap S_{2r+1,p}$$ is a bijection. By Lemma 4.3.2 we have
   $\varepsilon((-\Theta_{m-2r-1,d}) \bot pB)=(d,d_0)_p\varepsilon(pB),$ and $\varepsilon(dpB)=\varepsilon(pB)$  for $B \in p^{-1}S_{2r+1,p}(d_0d).$ Thus the assertion (1) follows from (H-$p$-3),(H-$p$-5) and Lemma 4.1.3.
   By (H-$p$-2) and Lemmas 4.1.3 and 4.3.2, we have
     $$Q^{(0)}(d_0,d,H_{m,\xi}^{(0)},2r,\varepsilon^l,t)={(1+p^{-(m-2r)/2}\chi(d)) ((-1)^{(m-2r)/2}d,d_0)_p^l \over 2\phi_{(m-2r)/2}(p^{-2})}$$
  $$\times Q^{(0)}(d_0d,1,H_{2r,\xi \chi(d)}^{(0)},2r,\varepsilon^l,t).$$ 
  Thus the assertion (2) immediately  follows in case $l\nu(d_0)=0.$ 
  Now suppose that $l=1$ and $\nu(d_0)=1.$  Take an element $a \in {\bf Z}_p^*$ such that $(a,p)_p=-1.$ Then the mapping $S_{2r}({\bf Z}_p) \ni B \mapsto aB \in S_{2r}({\bf Z}_p)$ induces a bijection from $S_{2r,p}(dd_0)$ to itself, and 
  $\varepsilon(apB)=-\varepsilon(pB)$ and $\alpha_p(apB)=\alpha_p(pB)$ for $B \in S_{2r,p}(dd_0).$ Furthermore by (H-$p$-5) we have 
 $$Q^{(0)}(d_0d,1,H_{2r,\xi \chi(d)}^{(0)},2r,\varepsilon^l,t)=\sum_{B \in S_{2r}(dd_0)} { H^{(0)}_{2r,\xi\chi(d)}(apB) \varepsilon(apB) \over \alpha_p(apB)}$$
 $$=-Q^{(0)}(d_0d,1,H_{2r,\xi \chi(d)}^{(0)},2r,\varepsilon^l,t).$$    Hence  $Q^{(0)}(d_0d,1,H_{2r,\xi \chi(d)}^{(0)},2r,\varepsilon^l,t)=0.$ This proves the assertion.
   
   Next suppose that $p=2.$   First suppose that $l=0,$ or $l=1$ and $d_0 \equiv 1 \ {\rm mod} \ 4.$   
 Fix a complete set ${\mathcal B}$  of representatives of $(S_{2r+2,2}(d_0)_o)/\approx.$ For $B \in {\mathcal B},$ let ${\mathcal S}_{2r+1,2}(B)$ and $\widetilde {\mathcal S}_{2r+2,2}(B)$ be those defined in Subsection 4.1.
Then, by (H-2-1) and (H-2-5) we have  
\begin{eqnarray*}
\lefteqn{
Q^{(0)}(d_0,H_{m,\xi}^{(0)},2r+1,\iota,t)} \\
&=&\sum_{B \in {\mathcal B}} {H_{2r+2,\xi}^{(0)}(2B) \over \phi_{(m-2r-2)/2}(2^{-2})2^{(r+1)(2r+3)}\alpha_2(B)}\#(\widetilde {\mathcal S}_{2r+2,2}(B)/\sim) t^{\nu(\det (2B))}.
\end{eqnarray*}
We have  $S_{2r+1,2}(d_0)=\cup_{B \in {\mathcal B}}{\mathcal S}_{2r+1,2}(B),$ and for any $B' \in {\mathcal S}_{2r+1,2}(B),$ we have $1 \bot B' \approx B.$ Hence $\nu(\det (2B))=\nu(\det (4B'))-2r$ and  $H_{2r+2,\xi}^{(0)}(2B)=H_{2r+2,\xi}^{(0)}(2 \bot 2B')=H_{2r+1,\xi}^{(1)}(4B').$ 
  Hence by Lemma 4.1.5 we have 
\begin{eqnarray*}
\lefteqn{Q^{(0)}(d_0,H_{m,\xi}^{(0)},2r+1,\varepsilon^l,t)} \\
&=&
2\sum_{B' \in S_{2r+1,2}(d_0)} {H_{2r+1,\xi}^{(1)}(4B')  \over 2^{(r+1)(2r+3)} \phi_{(m-2r-2)/2}(2^{-2})\alpha_2(B')} t^{\nu(\det (4B'))-2r} \\
&=&  2^{(2r+1)r}t^{-2r} \sum_{B' \in 2^{-1}S_{2r+1,2}(d_0) \cap S_{2r+1,2}} {H_{2r+1,\xi}^{(1)}(4B')  \over \phi_{(m-2r-2)/2}(2^{-2})\alpha_2(4B')} t^{\nu(\det (4B'))}.
\end{eqnarray*}
  This proves the assertion for $l=0.$ 
 Now let $d_0 \equiv 1 \ {\rm mod} \ 4,$ and  put $\xi_0=(2,d_0)_2.$ Then by Lemma 4.3.2 we have 
  $$\varepsilon( \Theta_{m-2r-2} \bot 2B)=(-1)^{m(m+2)/8+r(r+1)/2+(r+1)^2}
  \xi_0\varepsilon(B).$$
  Furthermore for any $a \in {\bf Z}_2^*$ we have $\varepsilon(aB)^l=\varepsilon(B)^l,$ and $\alpha_2(aB)=\alpha_2(B).$
  Thus, by using the same argument as above we obtain   
\begin{eqnarray*}
\lefteqn{Q^{(0)}(d_0,H_{m,\xi}^{(0)},2r+1,\varepsilon,t)=(-1)^{m(m+2)/8}\xi_0 } \\
&& \times \sum_{B \in {\mathcal B}} {H_{2r+2,\xi}^{(0)}(2B) (-1)^{m(m+2)/8+r(r+1)/2+(r+1)^2}\xi_0\varepsilon(B)  \over \phi_{(m-2r-2)/2}(2^{-2})2^{(r+1)(2r+3)}\alpha_2(B)}\#(\widetilde {\mathcal S}_{2r+2,2}(B)/\sim) t^{\nu(\det (2B))}.
\end{eqnarray*}
  We note that  $\varepsilon(1 \bot B')=\varepsilon(4B')$ for $B' \in S_{2r+1,2}.$ Hence, again  by Lemma 4.1.5, we have 
 $$Q^{(0)}(d_0,H_{m,\xi}^{(0)},2r+1,\varepsilon^l,t)=(-1)^{r(r+1)/2}((-1)^{r+1},(-1)^{r+1})_2 2^{(2r+1)r}t^{-2r}$$  
  $$ \times \sum_{B' \in S_{2r+1,2}(d_0)} {H_{2r+1,\xi}^{(1)}(4B') \varepsilon(B)  \over \phi_{(m-2r-2)/2}(2^{-2})\alpha_2(4B')} t^{\nu(\det (4B'))}.$$
This proves the assertion for $l=1$ and $d_0 \equiv 1 \ {\rm mod} \ 4.$

      Next  suppose that $l=1$ and $4^{-1}d_0 \equiv -1 \ {\rm mod} \ 4,$  or $l=1$ and $8^{-1}d_0 \in {\bf Z}_2^*.$ Then there exists an element $a \in {\bf Z}_2^*$ such that $(a,d_0)_2=-1.$  Then the map $2B \mapsto 2aB$ induces a bijection of $2S_{2r+2,2}(d_0)_o$ to itself. Furthermore  $H_{2r+2,\xi}^{(0)}(2aB)=H_{2r+2,\xi}^{(0)}(2B), \varepsilon(2aB)=-\varepsilon(2B),$ and $\alpha_2(2aB)=\alpha_2(2B).$ Thus the assertion can be proved by using the same argument as in the proof of (2) for $p \not=2.$ The  assertion (2) for $p=2$ can be proved by using (H-2-1), Lemmas 4.1.4 and 4.3.2 similarly to (2) for $p \not=2$.
\end{proof}


\begin{props}  
Let the notation be as above.     
\begin{enumerate}    
\item[{\rm (1)}] For $0 \le r \le (m-2)/2$ we have  
$$Q^{(1)}(d_0,H_{m-1,\xi}^{(1)},2r+1,\varepsilon^l,t)=  {Q^{(1)}(d_0, H_{2r+1,\xi}^{(1)},2r+1,\varepsilon^l,t) \over \phi_{(m-2r-2)/2}(p^{-2})}.$$
  
\item[{\rm (2)}]  For $1 \le r \le (m-2)/2$ and $d \in {\mathcal U}(m-1,m-2r-1,d_0)$ we have
   $$Q^{(1)}(d_0,d,H_{m-1,\xi}^{(0)},2r,\varepsilon^l,t)= {Q^{(0)}(d_0d,1,H_{2r,\xi \chi(d)}^{(0)},2r,\varepsilon^l,t) \over 2 \phi_{(m-2r-2)/2}(p^{-2})}$$
  if $l\nu(d_0)=0,$ and  
$$Q^{(1)}(d_0,d,H_{m-1,\xi}^{(0)},2r,\varepsilon^l,t) =0$$
otherwise. 
\end{enumerate}   

\end{props}
 
 \begin{proof} 
 We may suppose that $r <(m-2)/2.$ First suppose that $p \not=2.$
  As in the proof of Proposition 4.3.3 (1), we have a bijection $p^{-1}S_{2r+1,p}(d_0d) \cap S_{2r+1,p} \ni B \mapsto dB \in p^{-1}S_{2r+1,p}(d_0)  \cap S_{2r+1,p}.$  We also note that  $\varepsilon(dB)=\varepsilon(B),$ and $\alpha_p(dB)=\alpha_p(B).$   Hence, by (H-$p$-2), Lemmas 4.1.3 and 4.3.2, similarly to Proposition 4.3.3 (2), we have 
      $$Q^{(1)}(d_0,H_{m,\xi}^{(1)},2r+1,\varepsilon^l,t)=p^{(m/2-1)l\nu(d_0)}((-1)^{m/2} d_0,(-1)^{lm/2})_p$$
      $$ \times \sum_{B \in p^{-1}S_{2r+1,p}(d_0) \cap S_{2r+1,p}} {H_{2r+1,\xi}^{(1)}(pB)\varepsilon(pB)^l  \over 2\phi_{(m-2r-2)/2}(p^{-2})\alpha_p(pB)}t^{\nu(\det (pB))}$$
  $$ \times  \sum_{d \in {\mathcal U}}(1+p^{-(m-2r-2)/2} \chi(d))((-1)^{(m-2r-2)/2}d,(-1)^{r+1}d_0d)_p^l.$$
  Thus the assertion clearly holds if $l\nu(d_0)=0.$ Suppose that $l=1$ and $\nu(d_0)=1.$ Then  
\begin{eqnarray*}
\lefteqn{((-1)^{(m-2r-2)/2}d,(-1)^{r+1}d_0d)_p} \\
&=&\chi(d)((-1)^{r+1},(-1)^{r+1}d_0d)_p((-1)^{m/2},(-1)^{m/2}d_0)_p, 
\end{eqnarray*}
  and therefore
\begin{eqnarray*}
\lefteqn{\sum_{d \in {\mathcal U}}(1+p^{-(m-2r-2)/2} \chi(d))((-1)^{(m-2r-2)/2}d,(-1)^{r+1}d_0)_p } \\
&=&2p^{-(m-2r-2)/2}((-1)^{r+1},(-1)^{r+1}d_0d)_p((-1)^{m/2},(-1)^{m/2}d_0)_p. \end{eqnarray*}
This proves the assertion (1).
  
 By (H-$p$-4) and by Lemmas 4.1.3 and 4.3.2, we have 
 \begin{eqnarray*}
\lefteqn{Q^{(1)}(d_0,d,H_{m-1,\xi}^{(1)},2r,\varepsilon^l,t)} \\
&=& {Q^{(0)}(d_0d,1,H_{2r,\xi\chi(d)}^{(0)},2r,\varepsilon^l,t) \over  2\phi_{(m-2r-2)/2}(p^{-2})}((-1)^{(m-2r)/2}d,d_0)_p^l. 
\end{eqnarray*}
  Thus the assertion (2) immediately follows if $l\nu(d_0)=0.$ The assertion  for $l=1$ and $\nu(d_0)=1$ follows from Proposition 4.3.3 (2).
  
  Next suppose that $p=2.$ 
    We have 
\begin{eqnarray*}
\varepsilon(2\Theta_{m-2r-2,d} \bot 4B)&=& (-1)^{m(m-2)/8}(-1)^{r(r+1)/2}((-1)^{m/2},(-1)^{m/2}d_0)_2 \\
&& \hspace*{5mm}\times ((-1)^{r+1},(-1)^{r+1}d_0d)_2(d_0,d)_2\,\,\varepsilon(4B) 
\end{eqnarray*}
for $d \in {\mathcal U}$ and  $B \in S_{2r+1,2}(dd_0).$ Thus, similarly to (1) for $p \not=2$, we obtain
\begin{eqnarray*}
\lefteqn{
Q^{(11)}(d_0,H_{m-1,\xi}^{(1)},2r+1,\varepsilon^l,t)=(-1)^{r(r+1)l/2}t^{-2r}((-1)^{r+1},(-1)^{r+1}d_0)_2^l
} \\
&\times& 2^{(m/2-1)l\nu(d_0)} \sum_{B \in S_{2r+1,2}(d_0)_e} {2^{r(2r+1)}H_{2r+1,\xi}^{(1)}(4B)\varepsilon(4B)^l  \over 2 \cdot 2^{m-2r-2}  \phi_{(m-2r-2)/2}(2^{-2})\alpha_2(4B)}t^{\nu(\det (4B))} \\
&\times& \sum_{d \in {\mathcal U}}(1+2^{-(m-2r-2)/2} \chi(d))(d,d_0)_2^l\\
&=& \sum_{d \in {\mathcal U}}(1+2^{-(m-2r-2)/2} \chi(d))(d,d_0)_2^l {Q^{(11)}(d_0,H_{2r+1,\xi}^{(1)},2r+1,\varepsilon^l,t) \over 2^{1+(m-2r-2)(1-l\nu(d_0)/2)}\phi_{(m-2r-2)/2}(2^{-2})}.
\end{eqnarray*}
In the same manner as above, we obtain 
\begin{eqnarray*}
\lefteqn{
Q^{(12)}(d_0,H_{m-1,\xi}^{(1)},2r+1,\varepsilon^l,t)=(-1)^{r(r+1)l/2}t^{-2r}((-1)^{r+1},(-1)^{r+1}d_0)_2^l} \\
&\times& 2^{(m/2-1)l\nu(d_0)}\sum_{B \in S_{2r+1,2}(d_0)_o} {2^{r(2r+1)}H_{2r+1,\xi}^{(1)}(4B)\varepsilon(4B)^l  \over 2^{m-2r-2}  \phi_{(m-2r-2)/2}(2^{-2})\alpha_2(4B)}t^{\nu(\det (4B))} \\
&=&  {Q^{(11)}(d_0,H_{2r+1,\xi}^{(1)},2r+1,\varepsilon^l,t) \over 2^{(m-2r-2)(1-l\nu(d_0)/2)}\phi_{(m-2r-2)/2}(2^{-2})}.
\end{eqnarray*}
 Furthermore we have 
 \begin{eqnarray*}\varepsilon(-1 \bot 2\Theta_{m-2r-4} \bot 4B)&=&(-1)^{m(m-2)/8}(-1)^{r(r+1)/2}((-1)^{m/2},(-1)^{m/2}d_0)_2 \\
&& \hspace*{5mm}\times ((-1)^{r+1},(-1)^{r+1}d_0)_2(2,d_0)_2\varepsilon(2B) 
\end{eqnarray*}
for $d \in {\mathcal U}$ and  $B \in S_{2r+2,2}(dd_0)_o.$ 
Hence   
\begin{eqnarray*}
\lefteqn{Q^{(13)}(d_0,H_{m-1,\xi}^{(1)},2r+1,\varepsilon^l,t)=(-1)^{r(r+1)l/2}t^{-2r-2}((-1)^{r+1},(-1)^{r+1}d_0)_2^l } \\
&&\times (2,d_0)_2^l \,\,2^{(m/2-1)l\nu(d_0)} \hspace*{-7mm}
  \sum_{B \in S_{2r+2,2}(d_0)_o} {H_{2r+2,\xi}^{(0)}(2B)\varepsilon(4B)^l  \over   \phi_{(m-2r-4)/2}(2^{-2})\alpha_2(2B)}t^{\nu(\det (4B))} \\
&=& (((-1)^{r+1}2,d_0)_2(-1)^{(r+1)(r+2)/2})^l 2^{(m/2-1)l\nu(d_0)} \\
&& \times \sum_{B \in S_{2r+2,2}(d_0)_o} {H_{2r+2,\xi}^{(0)}(2B)\varepsilon(2B)^l  \over   \phi_{(m-2r-4)/2}(2^{-2})\alpha_2(2B)}t^{\nu(\det (2B))} \\
 &=& {Q^{(0)}(d_0,H_{2r+2,\xi}^{(0)},2r+1,\varepsilon^l,t) \over  \phi_{(m-2r-4)/2}(2^{-2})}2^{(m/2-1)l\nu(d_0)}.
 \end{eqnarray*}
  First suppose that $l=0$ or $\nu(d_0)$ is even. Then  $(d,d_0)_2^l=1.$ Hence     $$Q^{(11)}(d_0,H_{m-1,\xi}^{(1)},2r+1,\varepsilon^l,t)+Q^{(12)}(d_0,H_{m-1,\xi}^{(1)},2r+1,\varepsilon^l,t)$$
  $$=  {Q^{(1)}(d_0,H_{2r+1,\xi}^{(1)},2r+1,\varepsilon^l,t) \over 2^{(m-2r-2)(1-\nu(d_0)l/2)}\phi_{(m-2r-2)/2}(2^{-2})}.$$ 
  Furthermore by Proposition 4.3.3 (2), we have    $$Q^{(13)}(d_0,H_{m-1,\xi}^{(1)},2r+1,\varepsilon^l,t)={Q^{(1)}(d_0,H_{2r+1,\xi}^{(1)},2r+1,\varepsilon^l,t) \over  \phi_{(m-2r-4)/2}(2^{-2})}$$
  if $l\nu(d_0)=0,$ 
  and    
$$Q^{(13)}(d_0,H_{m-1,\xi}^{(1)},2r+1,\varepsilon,t)=0$$
 if $4^{-1}d_0 \equiv -1 \ {\rm mod} \ 4.$ 
 Thus summing up these two quantities, we prove the assertion.
 Next suppose that $l=1$ and $\nu(d_0)=3.$ 
Then, we have 
 $$Q^{(13)}(d_0,H_{m-1,\xi}^{(1)},2r+1,\varepsilon,t)=0.$$
 We prove 
$$Q^{(12)}(d_0,H_{2r+1,\xi}^{(1)},2r+1,\varepsilon,t)=0.$$
If $r=0,$ then clearly $S_{2r+1,2}(d_0)_o$ is empty. Suppose that $r \ge 1.$ Then 
for $B \in  4S_{2r+1,2;o}$ take a canonical form $4c_{01} \bot 4B_1$ with $c_{01} \in {\bf Z}_2^*, B_1 \in  S_{2r,2},$ and
 put $B'=4c_{01} \bot 4\cdot 5B_1.$ Then, by Corollary to Lemma 4.3.1,  the mapping $B \mapsto B'$ induces a bijection from $4S_{2r+1,2}(d_0)_o/\sim$ to itself, and 
 $\varepsilon(B')=-\varepsilon(B).$ Then,  by (H-2-6), and Lemma 4.1.4 (3), we can prove the above equality  
in the same way as in the proof of (1) for $p \not=2.$. We also note that $\sum_{d \in {\mathcal U}}(1+2^{-(m-2r-2)/2} \chi(d))(d,d_0)_2=2^{1-(m-2r-2)/2}.$
This proves the assertion.


 The assertion (2) for $p =2$ can be proved in the same manner as in (2) for $p \not=2.$
\end{proof}

\bigskip

 \subsection{Proof of the main result  }
 
 \noindent
{ }

\bigskip

 In this section, we prove our main result. First we give an explicit formula for the power series of Koecher-Maass
 type.

   
\begin{thms}
Let  $d_0 \in {\mathcal F}_{p},$ and put $\xi_0=\chi(d_0).$ Then we have the following:
\begin{enumerate}
\item[{\rm (1)}] $P_{n-1}^{(1)}(d_0,\iota,X,t)=\displaystyle{(p^{-1}t)^{\nu(d_0)}(1-\xi_0 t^2p^{-5/2}) \over   \phi_{(n-2)/2}(p^{-2})(1-t^2p^{-2}X)(1-t^2p^{-2}X^{-1}) }$ \\
\item[] \hspace*{15mm}$\times \displaystyle{1 \over \prod_{i=1}^{(n-2)/2} (1-t^2p^{-2i-1}X)(1-t^2p^{-2i-1}X^{-1}) }$. \\[3mm]
\item[{\rm (2)}]
$P_{n-1}^{(1)}(d_0, \varepsilon,X,t)=\displaystyle{(p^{-1}t)^{\nu(d_0)}(1-\xi_0 t^2p^{-1/2-n}) \over \phi_{(n-2)/2}(p^{-2})}$ \\
\item[] \hspace*{15mm}$\times \displaystyle{1 \over \prod_{i=1}^{(n-2)/2} (1-t^2p^{-2i-1}X)(1-t^2p^{-2i-1}X^{-1}) }$.
\end{enumerate}
\end{thms}  

To prove the  above theorem, we define another formal power series. Namely,  for $l=0,1$  we define $K_{n-1}^{(1)}(d_0,\varepsilon^l,X,t)$ as 
\begin{eqnarray*}
\lefteqn{K_{n-1}^{(1)}(d_0,\varepsilon^l,X,t)=\kappa(d_0,n-1,l)^{-1} t^{\delta_{2,p}(2-n)} }\\
&& \times \sum_{B' \in {\mathcal L}_{n-1,p}^{(1)}(d_0)}{G_p^{(1)}(B',p^{-(n+1)/2}X) \varepsilon(B')^l \over \alpha_p(B')}X^{-\mathfrak{e}^{(1)}(B')} t^{\nu(\det B')}.
\end{eqnarray*}
\begin{props}
Let $d_0$ be as above.  Then,  we have
$$P_{n-1}^{(1)}(d_0,\omega,X,t)=\prod_{i=1}^{n-1}(1-t^2Xp^{i-n-1})^{-1} K_{n-1}^{(1)}(d_0,\omega,X,t).$$
\end{props}

\begin{proof} We note that $B'$ belongs to ${\mathcal L}_{n-1,p}^{(1)}(d_0)$ if $B$ belongs to ${\mathcal L}_{n-1,p}^{(1)}(d_0)$ and $\alpha_p(B',B) \not=0.$ Hence  by Lemma 4.2.2 for $\omega=\varepsilon^l$ with $l=0,1$ we have
\begin{eqnarray*}
\lefteqn{P_{n-1}^{(1)}(d_0,\omega,X,t)=\kappa(d_0,n-1,l)^{-1}t^{\delta_{2,p}(2-n)} } \\
&& \times \sum_{B \in {\mathcal L}_{n-1,p}^{(1)}(d_0)}{1 \over \alpha_p(B)}\sum_{B'} {G_p^{(1)}(B',p^{-(n+1)/2}X) X^{-\mathfrak{e}^{(1)}(B')}\alpha_p(B',B) \omega(B')\over \alpha_p(B')} \\
&& \times (p^{-1}X)^{(\nu(\det B)-\nu(\det B'))/2}t^{\nu(\det B)} \\
&=&\kappa(d_0,n-1,l)^{-1}t^{\delta_{2,p}(2-n)}\sum_{B' \in {\mathcal L}_{n-1,p}^{(1)}(d_0)}{G_p^{(1)}(B',p^{-(n+1)/2}X) \omega(B') \over \alpha_p(B')}X^{-\mathfrak{e}^{(1)}(B')} \\
&&\times \sum_{B \in {\mathcal L}_{n-1,p}^{(1)}(d_0)} {\alpha_p(B',B) \over \alpha_p(B)} (p^{-1}X)^{(\nu(\det B)-\nu(\det B'))/2} t^{\nu(\det B)}.
\end{eqnarray*}
Hence  by Theorem 5 of \cite{B-S}, and by (1) of Lemma 4.1.1, we have 
\begin{eqnarray*}
\lefteqn{\sum_{B} {\alpha_p(B',B) \over \alpha_p(B)} (p^{-1}X)^{(\nu(\det B)-\nu(\det B'))/2} t^{\nu(\det B)}} \\
&=&\sum_{W \in M_{n-1}({\bf Z}_p)^{\times}/GL_{n-1}({\bf Z}_p)} (t^2Xp^{-1}p^{-n+1})^{\nu(\det W)}t^{\nu(\det B')} \\
&=&\prod_{i=1}^{n-1}(1-t^2Xp^{i-n-1})^{-1}t^{\nu(\det B')}. 
\end{eqnarray*}
Thus the assertion holds.
\end{proof}

For a variable $X$ we introduce the symbol $X^{1/2}$ so that $(X^{1/2})^2=X,$ and for an integer $a$ write $X^{a/2}=(X^{1/2})^a.$ Under this convention, we can write $X^{-\mathfrak{e}^{(1)}(T)}t^{\nu(\det T)}$ as $X^{\delta_{2,p}(n-2)/2}X^{\nu(d_0)/2}(X^{-1/2}t)^{\nu(\det T)}$ if $T \in {\mathcal L}_{r-1,p}'(d_0),$
and hence we can write $K_{n-1}^{(1)}(d_0,\varepsilon^l,X,t)$  as
 $$K_{n-1}^{(1)}(d_0,\varepsilon^l,X,t)=\kappa(d_0,n-1,l)^{-1}(tX^{-1/2})^{\delta_{2,p}(2-n)}X^{\nu(d_0)/2} $$
$$ \times \sum_{B' \in {\mathcal L}_{n-1,p}^{(1)}(d_0)}{G_p^{(1)}(B',p^{-(n+1)/2}X) \varepsilon(B')^l \over \alpha_p(B')}(tX^{-1/2})^{\nu(\det B')}.$$
In order to prove Theorem 4.4.1, 
we introduce some power series.  
Let $m$ be an  integer and $l=0$ or $1$. Then for $d_0 \in {\bf Z}_p^{\times}$ put 
\[
\zeta_{m}(d_{0}, \varepsilon^l,u )=  \sum_{T \in S_{m,p}(d_0)/\sim  
}
\frac{\varepsilon(T)^l}{\alpha_{p}(T)}u^{\nu(\det T)},
\]
and for $d_0 \in {\bf Z}_2^{\times}$ put
\[\zeta^*_{m}(d_{0},\varepsilon^l,u)  = 
\sum_{
T \in  S_{m,2}(d_0)_e/\sim}
\frac{\varepsilon(T)^l}{\alpha_{2}(T)}u^{\nu(\det T)}.
\]
We make the convention that 
 $\zeta_{0}(d_0,\varepsilon^l,u) =\zeta^*_{m}(d_{0},\varepsilon^l,u)=1$  or $0$ according as  $d_0 \in {\bf Z}_p^*$ or not. 
Now for $d \in {\bf Z}_p^{\times},$ let $Z_{m}(u,\varepsilon^l,d)$ and $Z_m^*(u,\varepsilon^l,d)$ be the formal power series in Theorems 5.1, 5.2, and 5.3 of \cite{I-S}, which are given by
$$Z_{m}(u,\varepsilon^l,d)=2^{-\delta_{2,p} m}\sum_{i=0}^{\infty} \sum_{T \in {\bf S}_{m}({\bf Z}_p,p^id)/\sim} {\varepsilon(T)^l \over \alpha_p(T)} (\eta_m^l  p^{(m+1)/2}u)^i$$
and 
$$Z_m^*(u,\varepsilon^l,d)=2^{-m}\sum_{i=0}^{\infty} \sum_{T \in {\bf S}_{m}({\bf Z}_2,2^id)_e/\sim} {\varepsilon(T)^l \over \alpha_2(T)} (\eta_m^l 2^{(m+1)/2}u)^i, $$
where ${\bf S}_{m}({\bf Z}_p,a)=\{T \in S_m({\bf Z}_p) \, | \, \det T=a \ {\rm mod} \ {{\bf Z}_p^*}^{\Box} \}, \ {\bf S}_{m}({\bf Z}_p,a)_e={\bf S}_{m}({\bf Z}_p,a) \cap S_{m}({\bf Z}_p)_e,$ and $\eta_m=((-1)^{(m+1)/2},p)_p$ or $1$ according as $m$ is odd or even. 
Here we recall that the  local density for $T \in S_m({\bf Z}_p)$  in our paper is $2^{-\delta_{2,p}m}$ times 
that in \cite{I-S}. 
Put 
$$Z_{m,e}(u,\varepsilon^l,d)={1 \over 2} (Z_{m}(u,\varepsilon^l,d)+Z_{m}(-u,\varepsilon^l,d)),$$
$$Z_{m,o}(u,\varepsilon^l,d)={1 \over 2} (Z_{m}(u,\varepsilon^l,d)-Z_{m}(-u,\varepsilon^l,d)).$$
We also define $Z_{m,e}^*(u,\varepsilon^l,d)$ and $Z_{m,o}^*(u,\varepsilon^l,d)$ in the same way. Furthermore put $x(i)=e$ or $o$ according as $i$ is even or odd. Let $d_0 \in {\mathcal F}_p.$ Let $p \not=2.$  Then 
$$\zeta_{m}(d_0,\varepsilon^l,u)=Z_{m,x(\nu(d_0))}(p^{-(m+1)/2}((-1)^{(m+1)/2},p)_p u,\varepsilon^l,p^{-\nu(d_0)}(-1)^{(m+1)/2}d_0)$$
or 
$$\zeta_m(d_0,\varepsilon^l,u)=Z_{m,x(\nu(d_0))}(p^{-(m+1)/2}u,\varepsilon^l,p^{-\nu(d_0)}(-1)^{[(m+1)/2]}d_0)$$
according as $m$ is odd and $l=1,$ or not. Let $p =2,$ and $m$ is odd. 
Then 
$$\zeta_m(d_0,\varepsilon^l,u)=2^{m}Z_{m,x(\nu(d_0))}(2^{-(m+1)/2}u,\varepsilon^l,2^{-\nu(d_0)}(-1)^{(m+1)/2}d_0).$$
Let $p =2$ and $m$ is even. Then 
$$\zeta_m^*(d_0,\varepsilon^l,u)=2^{m}Z_{m,x(\nu(d_0))}^*(2^{-(m+1)/2}u,\varepsilon^l,(-1)^{m/2}2^{-\nu(d_0)}d_0).$$
\begin{props}
  Let $d_0 \in {\mathcal F}_p.$ For a positive even integer $r$ and $ d \in {\mathcal U}$  put
$$c(r,d_0,d,X)=
(1-\chi(d_0)p^{-1/2}X)\prod_{i=1}^{r/2-1}(1-p^{2i-1}X^2)(1+\chi(d)p^{r/2-1/2}X),$$
and put  $c(0,d_0,d,X)=1.$ Furthermore, for a positive odd integer $r$ put
$$c(r,d_0,X)=(1-\chi(d_0)p^{-1/2}X)\prod_{i=1}^{(r-1)/2}(1-p^{2i-1}X^2).$$
\noindent
{\rm (1)} Suppose that $p\neq 2.$ 
\begin{enumerate}
\item[{\rm (1.1)}] Let $l=0$ or $\nu(d_0)=0.$ Then  
\begin{eqnarray*}
\lefteqn{K_{n-1}^{(1)}(d_0,\varepsilon^l,X,t) } \\
&=&X^{\nu(d_0)/2}\left\{\sum_{r=0}^{(n-2)/2}\sum_{d \in {\mathcal U}(n-1,n-2r-1,d_0)}{p^{-r(2r+1)}(tX^{-1/2})^{2r}c(2r,d_0,d,X) \over 
2^{1-\delta_{0,r}}\phi_{(n-2r-2)/2}(p^{-2})}\right. \\[2mm]
&& \hspace*{5mm} \times (p,d_0d)^l_p\,\zeta_{2r}(d_0d,\varepsilon^l,tX^{-1/2}) \\[2mm]
&&\hspace*{-3mm}+\left.\sum_{r=0}^{(n-2)/2}{p^{-(r+1)(2r+1)}(tX^{-1/2})^{2r+1}c(2r+1,d_0,X) \over \phi_{(n-2r-2)/2}(p^{-2})} \zeta_{2r+1}(p^*d_0,\varepsilon^l,tX^{-1/2})\right\},
\end{eqnarray*}
where $p^*d_0=pd_0$ or $p^{-1}d_0$ according as $\nu(d_0)=0$ or $\nu(d_0)=1.$ 

\item[{\rm (1.2)}] Let $\nu(d_0) =1.$ Then  
\begin{eqnarray*}
\lefteqn{K_{n-1}^{(1)}(d_0,\varepsilon,X,t)} \\
&=&X^{\nu(d_0)/2}\sum_{r=0}^{(n-2)/2}{p^{-(r+1)(2r+1)-r}(tX^{-1/2})^{2r+1}c(2r+1,d_0,X) \over \phi_{(n-2r-2)/2}(p^{-2})} \\
&& \times \zeta_{2r+1}(p^{-1}d_0,\varepsilon,tX^{-1/2}). 
\end{eqnarray*}
\end{enumerate}
\noindent
{\rm (2)} Suppose that $p=2.$ 
\begin{enumerate}
\item[{\rm (2.1)}] Let $l=0$ or $d_0 \equiv 1 \ {\rm mod} \ 4.$ Then
\begin{eqnarray*}
\lefteqn{K_{n-1}^{(1)}(d_0,\varepsilon^l,X,t)} \\
&=& X^{\nu(d_0)/2}\{\sum_{r= 0}^{(n-2)/2} \sum_{d \in {\mathcal U}(n-1,n-2r-1,d_0)}(tX^{-1})^{2r}2^{-r(2r+1)} { c(2r,d_0,d,X) \over 2^{1-\delta_{0,r}}\phi_{(n-2r-2)/2}(2^{-2})} \\[2mm]
&&\hspace*{10mm}  \times ((-1)^{(r+1)r/2} (2,d_0d)_2)^l \zeta_{2r}^*(d_0d,\varepsilon,tX^{-1/2}) \\[2mm]
&&\hspace*{5mm} +\sum_{r=0}^{(n-2)/2}(tX^{-1/2})^{2r+1}2^{-(r+1)(2r+1)}{c(2r+1,d_0,X) \over \phi_{(n-2r-2)/2}(2^{-2})} \\[2mm]
&&\hspace*{10mm} \times ((-1)^{(r+1)r/2} ((-1)^{r+1},(-1)^{r+1}d_0)_2)^l \zeta_{2r+1}(d_0,\varepsilon^l,tX^{-1/2})\}.
\end{eqnarray*}

\item[{\rm (2.2)}] Suppose that $4^{-1}d_0 \equiv -1 \ {\rm mod} \ 4$ or $8^{-1}d_0 \in {\bf Z}_2^*.$ Then  
\begin{eqnarray*}
\lefteqn{K_{n-1}^{(1)}(d_0,\varepsilon,X,t)} \\
&=&X^{\nu(d_0)/2}\sum_{r =0}^{(n-2)/2}(tX^{-1/2})^{2r+1}2^{-(r+1)(2r+1)-r\nu(d_0)}{c(2r+1,d_0,X) \over \phi_{(n-2r-2)/2}(2^{-2})}  \\
&& \times  (-1)^{(r+1)r/2} ((-1)^{r+1},(-1)^{r+1}d_0)_2 \zeta_{2r+1}(d_0,\varepsilon,tX^{-1/2}).
\end{eqnarray*}
\end{enumerate}
\end{props}
 
 \begin{proof} 
Put $H_{2r+j,\xi}^{(j)}(B)=1$ for $ j\in\{0,1\}, 1-j \le r \le m/2-j,\xi=\pm 1,$ and $B \in S_{2r+j,p}.$ Then clearly the set $\{ H_{2r+j,\xi}^{(j)} \, | \ j \in \{0,1\}
, 1-j \le r \le n/2-j,\xi=\pm 1 \}$ satisfy the conditions (H-p-0), $\cdots$, (H-p-5) in Subsection 4.3 for any positive even integer $m \le n$. Hence by Lemma 4.2.1 and Proposition 4.3.4, and by using the same argument as in Lemma 3.1 (1) of \cite{I-K3}, we have 
 \begin{eqnarray*}
 \lefteqn{K_{n-1}^{(1)}(d_0,\varepsilon^l,X,t)} \\
&=&\gamma_{l,d_0} X^{\nu(d_0)/2}  \sum_{r=0}^{(n-2)/2}\sum_{d \in {\mathcal U}(n-1,n-2r-1,d_0)}{c(2r,d_0,d,X) \over 2^{1-\delta_{0,r}}\phi_{(n-2r-2)/2}(p^{-2})} \\
&& \times \sum_{B \in S_{2r,p}(d_0d) } {\varepsilon(pB)^l \over \alpha_p(pB)}
(tX^{-1/2})^{\nu(\det (pB))} \\
&+&X^{\nu(d_0)/2}\sum_{r=0}^{(n-2)/2}{c(2r+1,d_0,d,X) \over \phi_{(n-2r-2)/2}(p^{-2})} \\
&& \times \sum_{B \in p^{-1}S_{2r+1,p}(d_0) \cap S_{2r+1,p}} {((-1)^{(r+1)/2},(-1)^{(r+1)/2}d_0)_p^l p^{-l\nu(d_0)}\varepsilon(pB)^l \over \alpha_p(pB)} \\
&& \times (tX^{-1/2})^{\nu(\det (pB))},
\end{eqnarray*}
where $\gamma_{l,d_0}=1$ or $0$ according as $\nu(d_0)l=0$ or $1.$ 
Thus the assertion (1.1) follows from Lemmas 4.1.3 and 4.3.2 by remarking that
$p^{-1}S_{2r+1,p}(d_0) \cap S_{2r+1,p}
= S_{2r+1}(p^{*}d_0).$ Similarly the assertion (1.2) can be proved by remarking that $\varepsilon(pB)=((-1)^{r+1},p)\varepsilon(B)$ for $B \in p^{-1}S_{2r+1,p}(d_0) \cap S_{2r+1,p}.$ 
The assertion for $p=2$ can also be proved in the same manner as above.
\end{proof}

\bigskip

\noindent
{\bf Remark.} As seen above, to prove Proposition 4.4.3, we have only 
to prove Propositions 4.3.2 and 4.3.3 for 
the simplest case where $\{H_{2r+j,\xi}^{(j)}\}$ are constant functions. However 
a similar statement for more general $\{H_{2r+j,\xi}^{(j)}\}$ 
will be necessary for giving an explicit formula for the Rankin-Selberg series of $\sigma_{n-1}(\phi_{I_n(h),1})$ (cf. \cite{K-K}). Indeed, 
the proofs are essentially the same as those for 
the simplest case.
This is why we formulate and prove those propositions in more general settings.  

 \begin{proof}[{\bf Proof of Theorem 4.4.1 in case $p\not=2$.}]
(1)  First let $d_0 \in {\bf Z}_p^*.$ Then by Proposition 4.4.3 (1.1), we have 
\begin{eqnarray*}
\lefteqn{K_{n-1}^{(1)}(d_0,\iota,X,t)={1 \over \phi_{(n-2)/2}(p^{-2})} } \\
&&+\sum_{r=1}^{(n-2)/2}\sum_{d \in {\mathcal U}}{p^{-r(2r+1)}(t^2X^{-1})^r\prod_{i=1}^{r-1}(1-p^{2i-1}X^2) \over 2\phi_{(n-2r-2)/2}(p^{-2})} \\[2mm]
&& \hspace*{10mm}\times (1-p^{-1/2}\xi_0 X)(1+\eta_d p^{r-1/2}X)\zeta_{2r}(d_0d,\iota,tX^{-1/2}) \\[2mm]
&&+\sum_{r=0}^{(n-2)/2}{p^{-(2r+1)(r+1)}(t^2X^{-1})^{r+1/2}\prod_{i=1}^{r}(1-p^{2i-1}X^2) \over \phi_{(n-2r-2)/2}(p^{-2})} \\[2mm]
&& \hspace*{10mm}\times (1-p^{-1/2}\xi_0 X)\zeta_{2r+1}(pd_0,\iota, tX^{-1/2}). 
\end{eqnarray*}
Here we put $\eta_d=\chi(d)$ for $d \in {\mathcal U}.$  
By Theorem 5.1 of \cite{I-S}, we have 
$$\zeta_{2r+1}(pd_0,\iota, tX^{-1/2})=\frac{p^{-1}t X^{-1/2}}{\phi_r(p^{-2})(1-p^{-2}t^2X^{-1})\prod_{i=1}^{r} (1-p^{2i-3-2r}t^2X^{-1})},$$
and
$$\zeta_{2r}(d_0d,\iota,t X^{-1/2})=\frac{(1+\xi_0 \eta_d p^{-r})(1-\xi_0 \eta_d p^{-r-2} t^2X^{-1})}{\phi_r(p^{-2})(1-p^{-2}t^2 X^{-1})\prod_{i=1}^{r} (1-p^{2i-3-2r}t^2 X^{-1})}.$$
Hence the assertion for $n=2$ can be proved by a direct calculation. Suppose that $n \ge 4.$ 
Then $K_{n-1}^{(1)}(d_0,\iota,X,t)$ can be expressed as
$$K_{n-1}^{(1)}(d_0,\iota,X,t)={S(d_0,\iota,X,t) \over \phi_{(n-2)/2}(p^{-2})(1-p^{-2}t^2 X^{-1})\prod_{i=1}^{(n-2)/2} (1-p^{2i-n-1}t^2 X^{-1})},$$
where $S(d_0,\iota,X,t)$ is a polynomial in $t$ of degree $n.$
 We have 
$$2^{-1} (1-p^{-1/2}\xi_0 X) \sum_{\eta =\pm 1} (1+\eta p^{(n-2)/2-1/2}X) (1+ \xi_0 \eta p^{-(n-2)/2})(1-\xi_0 \eta p^{-(n-2)/2-2}t^2 X^{-1})$$
$$=(1-\xi_0 p^{-1/2}X)(1+\xi_0p^{-1/2}X-\xi_0p^{-5/2}t^2-p^{-n}t^2X^{-1}).$$
Hence  
\begin{eqnarray*}
\lefteqn{2^{-1}\sum_{d \in {\mathcal U}}p^{(n-1)(-n+2)/2}(t^2X^{-1})^{(n-2)/2}\prod_{i=1}^{(n-2)/2-1}(1-p^{2i-1}X^2) } \\
&& \times (1-p^{-1/2}\xi_0 X)(1+\eta_d p^{(n-2)/2-1/2}X)\zeta_{n-2}(d_0d,\iota,tX^{-1/2}) \\
&+&p^{-(n-1)n/2}(t^2X^{-1})^{(n-2)/2+1/2}\prod_{i=1}^{(n-2)/2}(1-p^{2i-1}X^2)(p^{-2})^{-1} \\
&& \times (1-p^{-1/2}\xi_0 X)\zeta_{n-1}(pd_0,\iota, tX^{-1/2}) \\[2mm]
&=&{ (p^{-(n-1)}X^{-1}t^2)^{(n-2)/2} (1-\xi_0p^{-5/2}t^2)\prod_{i=0}^{(n-2)/2-1}(1-p^{2i-1}X^2) \over  \phi_{(n-2)/2}(p^{-2})(1-p^{-2}t^2 X^{-1})\prod_{i=1}^{(n-2)/2} (1-p^{2i-n-1}t^2 X^{-1})},
\end{eqnarray*}
and therefore $S(d_0,\iota,X,t)$ can be expressed as 
\def\theequation{\Alph{equation}} 
  \makeatletter
\begin{eqnarray}
\lefteqn{
S(d_0,\iota,X,t) 
} \\
&=& (p^{-(n-1)}X^{-1}t^2)^{(n-2)/2}  \prod_{i=0}^{(n-2)/2-1}(1-p^{2i-1}X^2)(1-p^{-5/2}\xi_0t^2) \nonumber \\
&&+(1-p^{-n+1}t^2X^{-1})U(X,t), \nonumber
\end{eqnarray}
where $U(X,t)$ is a polynomial in $X, X^{-1}$ and $t.$ Now by Proposition 4.4.2, we have 
\begin{eqnarray*}
\lefteqn{P_{n-1}^{(1)}(d_0,\iota,X,t)} \\
&&\hspace*{-5mm}={S(d_0,\iota,X,t) \over \phi_{(n-2)/2}(p^{-2})(1-p^{-2}t^2 X^{-1})\prod_{i=1}^{(n-2)/2} (1-p^{2i-n-1}t^2 X^{-1})\prod_{i=1}^{n-1}(1-p^{i-n-1}Xt^2) }.
\end{eqnarray*}
Hence the power series $P_{n-1}^{(1)}(d_0,\iota,X,t)$ is a rational function in $X$ and $t.$ Since we have $\widetilde F_p^{(1)}(T,X^{-1})=\widetilde F_p^{(1)}(T,X)$ for any $T \in {\mathcal L}_{n-1,p}^{(1)},$ we have $P_{n-1}^{(1)}(d_0,\iota,X^{-1},t)=P_{n-1}^{(1)}(d_0,\iota,X,t).$  This implies that the reduced denominator of  the rational function  $P_{n-1}^{(1)}(d_0,\iota,X,t)$ in $t$ is at most 
\[
(1-p^{-2}t^2 X^{-1})(1-p^{-2}t^2 X)\prod_{i=1}^{(n-2)/2} \{(1-p^{2i-n-1}t^2 X^{-1})(1-p^{2i-n-1}t^2 X)\}.
\]
Hence  we  have
\begin{equation}
S(d_0,\iota,X,t) =\prod_{i=1}^{(n-2)/2} (1-p^{2i-n-2}t^2 X)(a_0(X)+a_1(X)t^2) 
\end{equation}
with some polynomials $a_0(X),\,a_1(X)$ in $X+X^{-1}.$ We easily see $a_0(X)=1.$ By substituting  $p^{(n-1)/2}X^{1/2}$ for $t$ in (A) and (B), and comparing them we see $a_1(X)=-p^{-5/2}\xi_0.$ This proves the assertion.

Next let $d_0 \in p{\bf Z}_p^*.$ 
Then by Proposition 4.4.3 (1.1), we have 
\begin{eqnarray*}
\lefteqn{K_{n-1}^{(1)}(d_0,\iota,X,t) } \\
&=&X^{1/2} \,\{2^{-1}\sum_{r=1}^{(n-2)/2}\sum_{d \in {\mathcal U}}{p^{-r(2r+1)}(t^2X^{-1})^r\prod_{i=1}^{r-1}(1-p^{2i-1}X^2) \over \phi_{(n-2r-2)/2}(p^{-2})} \\[2mm]
&& \hspace*{15mm}\times (1+\eta_d p^{r-1/2}X)\zeta_{2r}(d_0d,\iota,tX^{-1/2}) \\[2mm]
&& \hspace*{10mm}+\sum_{r=0}^{(n-2)/2}{p^{-(2r+1)(r+1)}(t^2X^{-1})^{r+1/2}\prod_{i=1}^{r}(1-p^{2i-1}X^2) \over \phi_{(n-2r-2)/2}(p^{-2})} \\[2mm]
&& \hspace*{15mm}\times \zeta_{2r+1}(p^{-1}d_0, \iota,tX^{-1/2})\}.
\end{eqnarray*}
By Theorem 5.1 of \cite{I-S}, we have 
$$\zeta_{2r+1}(p^{-1}d_0, \iota,t X^{-1/2})=\frac{1}{\phi_r(p^{-2})(1-p^{-2}t^2X^{-1})\prod_{i=1}^{r} (1-p^{2i-3-2r}t^2X^{-1})},$$
and
$$\zeta_{2r}(d_0d,\iota,t X^{-1/2})=\frac{p^{-1}tX^{-1/2}}{\phi_{r-1}(p^{-2})(1-p^{-2}t^2 X^{-1})\prod_{i=1}^{r} (1-p^{2i-3-2r}t^2 X^{-1})}.$$
Thus the assertion can be proved in the same manner as above.

(2) First let $d_0 \in {\bf Z}_p^*.$ Then by Proposition 4.4.3 (1.1), we have 
\begin{eqnarray*}
\lefteqn{K_{n-1}^{(1)}(d_0,\varepsilon,X,t)={1 \over \phi_{(n-2)/2}(p^{-2})} } \\
&& +\sum_{r=1}^{(n-2)/2}\sum_{d \in {\mathcal U}}{p^{-r(2r+1)}(t^2X^{-1})^r\prod_{i=1}^{r-1}(1-p^{2i-1}X^2) \over 2\phi_{(n-2r-2)/2}(p^{-2})} \\[2mm]
&& \hspace*{5mm}\times (1-p^{-1/2}\xi_0 X)(1+\eta_d p^{r-1/2}X)\xi_0 \eta_d \zeta_{2r}(d_0d,\varepsilon, tX^{-1/2}) \\[2mm]
&& +\sum_{r=0}^{(n-2)/2}{p^{-(2r+1)(r+1)}(t^2X^{-1})^{r+1/2}\prod_{i=1}^{r}(1-p^{2i-1}X^2) \over \phi_{(n-2r-2)/2}(p^{-2})} \\[2mm]
&& \hspace*{5mm} \times (1-p^{-1/2}\xi_0 X) \zeta_{2r+1}(pd_0,\varepsilon,tX^{-1/2}).
\end{eqnarray*}
By Theorem 5.2 of \cite{I-S}, 
$$\zeta_{2r}(d_0d,\varepsilon,t X^{-1/2})=\frac{1+\xi_0 \eta_d p^{-r}}{\phi_{r}(p^{-2}) \prod_{i=1}^r(1-p^{-2i}t^2X^{-1})},$$
and 
$$\zeta_{2r+1}(pd_0,\varepsilon, t X^{-1/2})=\frac{p^{-r-1}tX^{-1/2}}{\phi_{r}(p^{-2}) \prod_{i=1}^{r+1}(1-p^{-2i}t^2X^{-1})}.$$
Hence  $K_{n-1}^{(1)}(d_0,\varepsilon,X,t)$ can be expressed as
$$K_{n-1}^{(1)}(d_0,\varepsilon,X,t)={T(d_0,\varepsilon,X,t) \over \phi_{(n-2)/2}(p^{-2})\prod_{i=1}^{n/2} (1-p^{-2i}t^2 X^{-1})},$$
where $T(d_0,\iota,X,t)$ is a polynomial in $t$ of degree $n,$ and expressed as
\begin{eqnarray}
\hspace*{4mm} T(d_0,\iota,X,t)
&=& (p^{-n}X^{-1}t^2)^{n/2}  (1-\xi_0p^{-1/2}X)\prod_{i=1}^{(n-2)/2}(1-p^{2i-1}X^2) 
\\ 
&&+(1-p^{-n}t^2X^{-1})V(X,t), \nonumber
\end{eqnarray}
with a polynomial $V(X,t)$ in $X, X^{-1}$ and $t.$ On the other hand, by using the same argument as (1), we can show that
\begin{equation}
T(d_0,\varepsilon,X,t) =\prod_{i=1}^{(n-2)/2} (1-p^{-2i-1}t^2 X)(1+b_1(X)t^2) 
\end{equation}
with $b_1(X)$ a polynomial in $X+X^{-1}.$  Thus, by substituting  $p^{n/2}X^{1/2}$ for $t$ in (C) and (D), and comparing them we prove the assertion. 

Next let $d_0 \in p{\bf Z}_p^*.$ Then by Proposition 4.4.3 (1.2), we have 
\begin{eqnarray*}
K_{n-1}^{(1)}(d_0,\varepsilon,X,t)
&=&X^{1/2}\sum_{r=0}^{(n-2)/2}{p^{-(2r+1)(r+1)-r}(t^2X^{-1})^{r+1/2}\prod_{i=1}^{r}(1-p^{2i-1}X^2) \over \phi_{(n-2r-2)/2}(p^{-2})} \\
&& \times \zeta_{2r+1}(p^{-1}d_0,\varepsilon,tX^{-1/2}).
\end{eqnarray*}
By Theorem 5.2 of \cite{I-S}, 
$$\zeta_{2r+1}(p^{-1}d_0,\varepsilon, tX^{-1/2})=\frac{1}{\phi_{r}(p^{-2}) \prod_{i=1}^r(1-p^{-2i}t^2X^{-1})}.$$
Hence 
\begin{eqnarray*}
K_{n-1}^{(1)}(d_0,\varepsilon,X,t)
&=&p^{-1}t\sum_{r=0}^{(n-2)/2}{p^{-(2r+1)r}(p^{-2}t^2X^{-1})^{r}\prod_{i=1}^{r}(1-p^{2i-1}X^2) \over \phi_{(n-2r-2)/2}(p^{-2})} \\
&& \times  \frac{1}{\phi_{r}(p^{-2}) \prod_{i=1}^r(1-p^{-2i}t^2X^{-1})}.
\end{eqnarray*}
Thus the assertion can be proved in the same way as above. 
\end{proof}

\begin{proof}
[{\bf Proof of Theorem 4.4.1 in case $p=2$.}]
The assertion can also be proved by Proposition 4.4.3 (2) in the same way as above.
\end{proof}
 
\begin{props}
Let $k$ and $n$ be positive even integers. Given a Hecke eigenform $ h \in \mathfrak{S}_{k-n/2+1/2}^+(\varGamma_0(4)),$  let $f \in \mathfrak{S}_{2k-n}(\varGamma^{(1)})$ be the primitive form as in Section 2. Then
\[
L(s, h)=L(2s,f)\sum_{d_0 \in {\mathcal F}^{(-1)^{n/2} } }c_h(|d_0|) |d_0|^{-s} L(2s-k+n/2+1, \left({d_0 \over *}\right))^{-1},
\]
where $L(s, \left({d_0 \over *}\right))$ is Dirichlet's $L$-function for the character  $\left({d_0 \over *} \right).$
\end{props}
\begin{proof}
The assertion can immediately be proved by remarking the fact that
$$\sum_{m=1}^{\infty} c_h(|d_0|m^2)m^{-2s}=c_h(|d_0|)L(2s-k+n/2+1, \left({d_0 \over *}\right))^{-1}L(2s,f)$$
for $d_0 \in {\mathcal F}^{(-1)^{n/2}}.$
\end{proof}
\begin{proof}
[{\bf Proof of Theorem 2.1.}]
By Theorem 4.4.1, we have
\begin{eqnarray*}
\lefteqn{
\prod_p P_{n-1,p}^{(1)}(d_0,\iota_p,\alpha_p,p^{-s+k/2+n/4-1/4})=|d_0|^{-s+k/2+n/4-5/4} }\\
&&\times \prod_{i=1}^{(n-2)/2} \zeta(2i)  L(2s-k-n/2+3,\left({d_0 \over *}\right))^{-1} \prod_{i=1}^{(n-2)/2}L(2s-n+2i+1,f),
\end{eqnarray*}
and
\begin{eqnarray*}
\lefteqn{
\prod_p P_{n-1,p}^{(1)}(d_0,\varepsilon_p,\alpha_p,p^{-s+k/2+n/4-1/4})
=|d_0|^{-s+k/2+n/4-5/4} }\\
&& \times  \prod_{i=1}^{(n-2)/2} \zeta(2i)L(2s-k+n/2+1,\left({d_0 \over *}\right))^{-1} \prod_{i=1}^{(n-2)/2}L(2s-n+2i,f).
\end{eqnarray*}
Thus the assertion follows from Theorem 3.2 and Proposition 4.4.4.
\end{proof}

\noindent
{\bf Remark.} Let $m$ be a nonnegative integer, and let $k$ be a positive integer such that $k > m+2$. Let $E_k^{(m+1)}$ be the Siegel Eisenstein series of weight $k$ and of degree $m+1$. (For the definition of the Siegel Eisenstein series, see, for example, \cite{Ha}.) 
Suppose that  $m >0$ and let $e_{k,1}^{(m+1)}$ be the first Fourier-Jacobi coefficient of $E_k^{(m+1)}.$ Then $e_{k,1}^{(m+1)}$ belongs to $J_{k,\, 1}(\varGamma_J^{(m)}).$
In \cite{Ha}, Hayashida defined the generalized Cohen Eisenstein series $E_{k-1/2}^{(m)}$ as $E_{k-1/2}^{(m)}=\sigma_m(e_{k,1}^{(m+1)}),$
where $\sigma_m$ is the Ibukiyama isomorphism.  It turns out that $E_{k-1/2}^{(m)}$ belongs to $\mathfrak{M}_{k-1/2}^{+}(\varGamma_0^{(m)}(4)),$ and in particular, $E_{k-1/2}^{(1)}$ 
coincides with the Cohen Eisenstein series defined 
 in \cite{Co}. 
Let $k$ and $n$ be positive even integers such that $k >n+1.$ Then, 
$E_{2k-n}^{(1)}$ is the Hecke eigenform corresponding to $E_{k-n/2+1/2}^{(1)}$ under the Shimura correspondence, and $E_k^{(n)}$ can be regarded
as a  non-cuspidal version of the Duke-Imamo{\=g}lu-Ikeda lift of $E_{k-n/2+1/2}^{(1)}.$ Therefore, by using the same method as in the proof of Theorem 2.1, we can express the Koecher-Maass series of $E_{k-1/2}^{(n-1)}$ explicitly in terms of $L(s,E_{k-n/2+1/2}^{(1)})$ and $L(s,E_{2k-n}^{(1)}).$

\subsection*{Acknowledgements}
The authors thank the referee for 
suggesting them the remark after the proof of Theorem 2.1.
The first named author was partly supported by Grant-in-Aid for Scientific Research, JSPS (No.\,21540004), and the second named author was partially supported by Grants-in-Aid for Scientific Research, JSPS (No.\,23224001 and No.\,22340001).

\end{document}